\documentclass[11pt]{amsart}
\usepackage{amsfonts,mathtools,amssymb,amsthm,tikz-cd,color,hyperref,enumitem}
\usepackage{graphicx,cleveref}
\usepackage[font=small]{caption}
\usepackage[normalem]{ulem}
\usepackage[twoside=false]{geometry}
\usepackage[backend=bibtex,
style=alphabetic,
autocite=inline,maxbibnames=100,sortcites,doi=false,isbn=false]{biblatex}
\renewbibmacro{in:}{}

\usetikzlibrary{decorations.pathreplacing}

\title{Braid groups, elliptic curves, and resolving the quartic}
\author{Peter Huxford and Jeroen Schillewaert}
\address{PH: Department of Mathematics, University of Chicago, 5734 S University Ave, Chicago, IL 60637, USA}
\email{pjhuxford@uchicago.edu}
\address{JS: Department of Mathematics, University of Auckland, 38 Princes Street, Auckland 1010, New Zealand}
\email{j.schillewaert@auckland.ac.nz}

\newcommand{\C}{\mathbb{C}}
\newcommand{\HH}{\mathbb{H}}
\newcommand{\M}{\mathcal{M}}
\newcommand{\R}{\mathbb{R}}
\newcommand{\Z}{\mathbb{Z}}
\newcommand{\Q}{\mathbb{Q}}
\newcommand{\D}{\mathbb{D}}

\newcommand{\CP}{\mathbb{CP}}

\newcommand{\PP}{\mathbb{P}}

\newcommand{\gap}{\,\rule[-0.1mm]{0.5cm}{0.15mm}\,}

\DeclareMathOperator{\Conf}{Conf}

\DeclareMathOperator{\PConf}{PConf}
\DeclareMathOperator{\SL}{SL}
\DeclareMathOperator{\GL}{GL}
\DeclareMathOperator{\PSL}{PSL}

\DeclareMathOperator{\Aff}{Aff}
\DeclareMathOperator{\Isom}{Isom}
\DeclareMathOperator{\orb}{orb}

\DeclareMathOperator{\Aut}{Aut}
\DeclareMathOperator{\Inn}{Inn}

\DeclareMathOperator{\Mod}{Mod}
\DeclareMathOperator{\crs}{crs}
\DeclareMathOperator{\ocrs}{\crs^\circ}

\DeclareMathOperator{\id}{id}
\DeclareMathOperator{\Out}{Out}

\numberwithin{equation}{section}

\newtheorem{theorem}[equation]{Theorem}
\newtheorem*{theorem*}{Theorem}
\newtheorem{lemma}[equation]{Lemma}

\newtheorem{proposition}[equation]{Proposition}

\theoremstyle{definition}
\newtheorem*{definition}{Definition}
\newtheorem{remark}[equation]{Remark}

\newtheorem*{question*}{Question}

\crefformat{section}{\S#2#1#3}
\crefmultiformat{section}{\S\S#2#1#3}{ and~#2#1#3}{, #2#1#3}{, and~#2#1#3}

\begin{document}

\maketitle

\begin{abstract}
  We show that, up to a natural equivalence relation, the only non-trivial, non-identity holomorphic maps $\Conf_n\C\to\Conf_m\C$ between unordered configuration spaces, where $m\in\{3,4\}$, are the resolving quartic map $R\colon\Conf_4\C\to\Conf_3\C$, a map $\Psi_3\colon\Conf_3\C\to\Conf_4\C$ constructed from the inflection points of elliptic curves in a family, and $\Psi_3\circ R$. This completes the classification of holomorphic maps $\Conf_n\C\to\Conf_m\C$ for $m\leq n$, extending results of Lin, Chen and Salter, and partially resolves a conjecture of Farb. We also classify the holomorphic families of elliptic curves over $\Conf_n\C$. To do this we classify homomorphisms between braid groups with few strands and $\PSL_2\Z$, then apply powerful results from complex analysis and Teichm{\"u}ller theory. Furthermore, we prove a conjecture of Castel about the equivalence classes of endomorphisms of the braid group with three strands.
\end{abstract}

\section{Introduction}

For a topological space $X$, let
\[
  \PConf_n X \coloneqq \{(x_1,\ldots,x_n)\in X^n : x_i\neq x_j \text{ for all } i\neq j \}
\]
be the \emph{ordered configuration space} of $n$ points in $X$. The symmetric group $S_n$ acts freely and properly discontinuously on $\PConf_nX$ by permuting coordinates, and the quotient
\[ \Conf_nX\coloneqq (\PConf_nX) / S_n \]
is the \emph{(unordered) configuration space} of $n$ points in $X$, and can be thought as the space of $n$-element subsets of $X$. If $X$ is a complex manifold, then so too are $\PConf_nX$ and $\Conf_nX$. In this paper we address some interesting special cases of the following fundamental question.

\begin{question*}
  \label{main-question}
  What are the holomorphic maps $\Conf_n\C\to\Conf_m\C$?
\end{question*}

An answer to this question can be thought of as a classification of the geometric constructions which take $n$ points in the complex plane, and produce $m$ points in the complex plane.

Lin answered this question when $n\geq5$ and $m\leq n$ \cite{Lin04a}, which was generalized by Chen and Salter to the cases with $n\geq5$ and $m\leq 2n$ \cite{CS23}. In order to state their results and our own, we first define a natural equivalence relation on the set such maps, whose name we take from Chen and Salter \cite{CS23}.

\begin{definition}[\textbf{Affine twist}]
  Let $f\colon Y\to Z$ be a holomorphic map between complex manifolds $Y$ and $Z$, and suppose the \emph{affine group}
  \[
    \Aff\coloneqq\{z\mapsto az+b : a\in\C^*, b\in\C \} \cong \C\rtimes\C^*.
  \]
  has a holomorphic action on $Z$. If $A\colon Y\to\Aff$ is a holomorphic map, then the \emph{affine twist} $f^A\colon Y\to Z$ of $f$ by $A$ is the holomorphic map given by
  \[
    f^A(y) \coloneqq A(y) \cdot f(y).
  \]
  We say that $f$ and $f^A$ are \emph{affine equivalent}.
\end{definition}

We let $\Aff$ act on $\Conf_n\C$ element-wise, and consider the holomorphic maps $\Conf_n\C\to\Conf_m\C$ up to affine equivalence.

Chen and Salter define the following class of examples of holomorphic maps $\Conf_n\C\to\Conf_m\C$ \cite[\S3]{CS23}.

\begin{definition}[\textbf{Root map}]
  If $k,p\geq1$ and $\epsilon\in\{0,1\}$, then a \emph{basic root map} is a map $r_{p,\epsilon}\colon\Conf_k\C^*\to\Conf_{kp+\epsilon}\C$ given by taking $p$th roots, and also including zero if $\epsilon=1$. A \emph{root map} $\Conf_n\C\to\Conf_{kp+\epsilon}$ is a composition of the form
  \[
    \Conf_n\C \to \Conf_k\C^* \xrightarrow{r_{p,\epsilon}} \Conf_{kp+\epsilon}\C,
  \]
  where the first map is an affine twist of a constant map $\Conf_n\C\to\Conf_k\C^*$. To make sense of the affine twist, we let $\Aff$ act on $\Conf_k\C^*$ element-wise with translations acting trivially.
\end{definition}

It's worth noting that constant maps are affine twists of root maps. In fact, Chen and Salter prove that affine twists of root maps are precisely the holomorphic maps $\Conf_n\C\to\Conf_m\C$ for which the induced map $\Conf_n\C\to\Conf_m\C/\Aff$ is constant \cite[\S3.3]{CS23}.

Now we are in a position to more precisely state what is known about holomorphic maps between unordered configuration spaces. Lin's results \cite[Theorems 1.4, 4.5, 4.8]{Lin04b} together with Chen and Salter's generalizations \cite[Theorem 3.1]{CS23} imply that if $n\geq5$ and $m\leq 2n$, then, up to affine equivalence, the only holomorphic maps $\Conf_n\C\to\Conf_m\C$ are root maps, and the identity map when $n=m$.

It is an open question whether all holomorphic maps $\Conf_n\C\to\Conf_m\C$ with $n\geq5$ are affine twists of either a root map or of the identity map. However, for $n=3,4$ there are many examples of holomorphic maps $\Conf_n\C\to\Conf_m\C$ that are not of this form.

The next map we describe is such an example, and arises from identifying a monic, square-free polynomial of degree $n$ over $\C$ with the set of its roots in $\Conf_n\C$. It has some interesting history: as an 18-year-old in 1540, Lodovico Ferrari discovered a method for solving a general quartic equation, by first solving a closely related cubic equation. This method gives rise to a holomorphic map often referred to as \emph{resolving the quartic}.

\begin{definition}[\textbf{Resolving the quartic}]
  Let $R\colon\Conf_4\C\to\Conf_3\C$ be the map that lifts to $\tilde{R}\colon\PConf_4\C\to\PConf_3\C$ given by
  \[
    \tilde{R}(x_1,x_2,x_3,x_4) \coloneqq (x_1x_4+x_2x_3,x_1x_3+x_2x_4,x_1x_2+x_3x_4).
  \]
\end{definition}

It is clear that $\tilde{R}$ defines a map $\PConf_4\C\to\C^3$, but it is quite remarkable that its image is contained in $\PConf_3\C$. Indeed, the cross ratio of the three points $\tilde{R}(x_1,x_2,x_3,x_4)$ and $\infty$ is equal to the cross ratio of $x_1,x_2,x_3$, and $x_4$.

Farb conjectured that, up to affine equivalence, any holomorphic map $\Conf_n\C\to\Conf_m\C$ with $n\geq4$ and $m\geq3$ is either a root map, the identity, or factors through the resolving quartic map $R\colon\Conf_4\C\to\Conf_3\C$ \cite[Conjecture 2.1]{Far23}. The aforementioned work of Lin \cite{Lin04a}, and Chen and Salter \cite{CS23} confirms this is the case in the ranges $n\geq5$, $m\leq 2n$.

Farb points out that this conjecture requires the assumption $n\geq4$. This is due to the existence of the following holomorphic maps. Consider the family of elliptic curves over $\Conf_3\C$ equipped with a planar embedding given by the equation
\begin{equation}
  \label{eq:elliptic-curves}
  y^2 = (x-x_1)(x-x_2)(x-x_3), \qquad \{x_1,x_2,x_3\}\in\Conf_3\C \tag{$\ast$}
\end{equation}
For $k\geq2$, let \emph{Jordan's totient function} $J_2(k)$ be the number of elementary $k$-torsion points on an elliptic curve, and let $m_k=J_2(k)/2$ if $k>2$, and $m_2=J_2(2)=3$.

\begin{definition}[\textbf{Elliptic curve construction}]
  For $k\geq2$, let $\Psi_k\colon\Conf_3\C\to\Conf_{m_k}\C$ be the holomorphic map given by
  \[
    \Psi_k(\{x_1,x_2,x_3\}) \coloneqq
    \left\{\begin{tabular}{l}
      $x$-coordinates of the elementary $k$-torsion \\[0.1em]
      points of $y^2=(x-x_1)(x-x_2)(x-x_3)$
    \end{tabular}\right\}.
  \]
\end{definition}

Although $\Psi_2\colon\Conf_3\C\to\Conf_3\C$ is just the identity map, one can show $\Psi_3\colon\Conf_3\C\to\Conf_4\C$ is a holomorphic map that does not lift to $\PConf_3\C\to\PConf_4\C$.

We're now in a position to state our main theorem. We complete the classification of holomorphic maps $\Conf_n\C\to\Conf_m\C$ for $m\leq n$. We also partially resolve Conjecture 2.1 of Farb \cite{Far23}.
 
\begin{theorem}
  \label{thm:main-theorem}
  If $n\geq3$ and $m\in\{3,4\}$, then, up to affine equivalence, every holomorphic map $\Conf_n\C\to\Conf_m\C$ is either
  \begin{itemize}
  \item A root map $\Conf_n\C\to\Conf_m\C$
  \item The identity map $\Conf_n\C\to\Conf_n\C$
  \item $R\colon\Conf_4\C\to\Conf_3\C$
  \item $\Psi_3\colon\Conf_3\C\to\Conf_4\C$
  \item $\Psi_3\circ R\colon\Conf_4\C\to\Conf_4\C$
  \end{itemize}
\end{theorem}

It's worth noting that the composition $R\circ\Psi_3\colon\Conf_3\C\to\Conf_3\C$ happens to be affine equivalent to a root map.

Combined with the results of Lin \cite{Lin04a}, and Chen and Salter \cite{CS23}, \Cref{thm:main-theorem} completes the classification of holomorphic maps $\Conf_n\C\to\Conf_m\C$ for $m\leq n$.

Our second theorem concerns geometric constructions of elliptic curves. Let $\M_{g,n}$ be the moduli space of genus $g$ Riemann surfaces with $n$ unordered marked points. Then $\M_{1,1}\cong\SL_2\Z\backslash\HH$ is the moduli space of elliptic curves. Let the \emph{hyperelliptic map}
\[
  H \colon \Conf_3\C \to \M_{1,1}
\]
be the holomorphic map corresponding to the family of elliptic curves \eqref{eq:elliptic-curves}, forgetting the planar embedding. We think of $\M_{1,1}$ as a complex orbifold, so that holomorphic maps $Y\to\M_{1,1}$, where $Y$ is a complex manifold, bijectively correspond to isomorphism classes of families of elliptic curves over $Y$. We may forget the orbifold structure of $\M_{1,1}$ by using the $j$-invariant $j\colon\M_{1,1}\to\C$. If $f,g\colon Y\to\M_{1,1}$ are holomorphic families of elliptic curves, then $j\circ f=j\circ g$ if and only if for all $y\in Y$ the fibers over $y$ in each family are isomorphic.

Let $\Delta\colon\Conf_n\C\to\C^*$ be the \emph{discriminant}
\[
  \Delta(\{x_1,\ldots,x_n\}) \coloneqq \prod_{i\neq j}(x_i-x_j).
\]
We can affine twist by the discriminant by viewing $\C^*$ as a subgroup of $\Aff$.

\begin{theorem}
  \label{thm:family-elliptic-curves}
  Let $f\colon\Conf_n\C\to\M_{1,1}$ be a holomorphic family of elliptic curves over $\Conf_n\C$. Let $\id^{\Delta}$ be the affine twist of the identity on $\Conf_3\C$ by the discriminant $\Delta$. Then one of the following holds:
  \begin{itemize}
  \item $j\circ f$ is constant,
  \item $n=3$ and $f$ is $H$ or $H\circ\id^\Delta$,
  \item $n=4$ and $f$ is $H\circ R$ or $H\circ\id^\Delta\circ R$.
  \end{itemize}
\end{theorem}

Although $j\circ H=j\circ H\circ\id^{\Delta}$, the families $H$ and $H\circ\id^{\Delta}$ have different monodromy. A similar phenomenon occurs in a theorem of Chen and Salter. They prove that if $\M_g\coloneqq\M_{g,0}$, then the hyperelliptic family $\Conf_n\C\to\M_{\lfloor(n-1)/2\rfloor}$ is the only non-trivial holomorphic family $\Conf_n\C\to\M_g$ for $n\geq26$ and $g\leq n-2$ up to precomposing with $\id^\Delta\colon\Conf_n\C\to\Conf_n\C$, see \cite[\S3.5]{CS23}.

There is also a hyperelliptic family of genus one curves over $\Conf_4\C$, given by the equation
\[
  y^2=(x-x_1)(x-x_2)(x-x_3)(x-x_4).
\]
By taking the Jacobian of the underlying genus one curves in this family, we obtain a family $\Conf_4\C\to\M_{1,1}$ of elliptic curves. One can show that this is isomorphic to the family $H\circ R$.

To prove \Cref{thm:main-theorem,thm:family-elliptic-curves} we first classify the homomorphisms between the relevant fundamental groups.

The \emph{braid group on $n$ strands of $X$}, where $X$ is a topological space, is the fundamental group
\[
  B_n(X) \coloneqq \pi_1(\Conf_nX).
\]

The \emph{braid group on $n$ strands} is $B_n\coloneqq B_n(\C)$. Lin classified the homomorphisms $B_n\to B_m$ for $n\geq5$ and $m\leq n$ \cite{Lin04b}. These results were extended by Castel who handled the cases $n\geq6$, $m\leq n+1$ \cite{Cas16}, and finally Chen, Kordek, and Margalit extended this to $n\geq5$, $m\leq 2n$ \cite{CKM23}.

Chen, Kordek, and Margalit's classification is in terms of a natural equivalence relation on the set of homomorphisms $B_n\to B_m$ that we define in \Cref{sec:equivalence-classes}. This equivalence relation is roughly analogous to affine equivalence for holomorphic maps. They show that if $n\geq5$ and $m\leq 2n$, then there are at most five equivalence classes of homomorphisms $B_n\to B_m$, and they give explicit representatives of these equivalence classes \cite[Theorem 1.1]{CKM23}.

We prove the following theorem which stands in stark contrast to these results. The special case $m=n=3$ is a conjecture of Castel \cite[Conjecture 14.2]{Cas09}.

\begin{theorem}
  \label{thm:equivalence-classes}
  If $n\in\{3,4\}$ and $m\geq3$, then there are infinitely many equivalence classes of homomorphisms $B_n\to B_m$.
\end{theorem}

The ubiquity of homomorphisms in these cases reflects that previously established classification methods break down in the cases we consider, both in the group theoretic and holomorphic settings.

Indeed, Chen and Salter's classification \cite[Theorem 3.1]{CS23} uses techniques from Teichm{\"u}ller theory and complex analysis to show that at most two of Chen, Kordek, and Margalit's five equivalence classes can be represented by holomorphic maps. Since both of these are in fact represented by holomorphic maps, they are able to apply a theorem of Imayoshi and Shiga \cite{IS88} to complete their classification.

Our proofs require, in addition to the techniques of Chen and Salter, a well-known consequence of work of Bers \cite{Ber78}, which concerns the monodromy of families of Riemann surfaces over a punctured disk, see \Cref{thm:punctured-disk-family}. This fact imposes a constraint on homomorphisms $B_n\to B_m$ that can arise from a holomorphic map. When $n\geq5$ this constraint is automatically satisfied for all known homomorphisms with non-cyclic image. However, when $n\in\{3,4\}$ there are infinitely many inequivalent examples where it is not satisfied.

To prove \Cref{thm:main-theorem} we classify the homomorphisms $B_n\to B_m$ for $m\in\{3,4\}$ that satisfy the constraint imposed by \Cref{thm:punctured-disk-family} and the other constraints used by Chen and Salter \cite{CS23}. Our methods rely on special properties of the braid groups $B_3$ and $B_4$. To handle homomorphisms $B_n\to B_4$ we solve a delicate equation in the free group of rank 2 --- this is the sole purpose of \Cref{sec:f2}.

To prove \Cref{thm:family-elliptic-curves}, many of the considerations for \Cref{thm:main-theorem} also apply, since the orbifold fundamental group $\pi_1^{\orb}(\M_{1,1})\cong\SL_2\Z$ is a central quotient of $B_3$.

\subsection{Structure of the paper}

In \Cref{sec:homfromBn} we classify homomorphisms $B_3\to\PSL_2\C$, and use this to classify homomorphisms $B_n\to\PSL_2\Z$ and $B_n\to B_3$. We discuss Nielsen--Thurston theory in \Cref{sec:Nielsen-Thurston} and use it in \Cref{sec:equivalence-classes} to prove \Cref{thm:equivalence-classes}, in particular proving a conjecture of Castel.
In \Cref{sec:B3toB3andB4} we use the Nielsen--Thurston classification theorem to classify the homomorphisms $B_3\to B_4$ and $B_4\to B_4$. In \Cref{sec:f2} we prove a property of a particular automorphism of the free group of rank 2. We use many of our group theoretic results, together with results on holomorphic maps we discuss in \Cref{sec:conf-spaces}, to prove \Cref{thm:main-theorem,thm:family-elliptic-curves}.

\subsection{Acknowledgements}

We especially thank Benson Farb for suggesting the use of the Nielsen--Thurston classification theorem, pointing out some important applications of Teichm{\"u}ller theory, and for extensive comments. We also thank Lei Chen and Nick Salter whose work inspired this paper. Furthermore, we thank all the previously mentioned people as well as Ishan Banerjee, Rose Elliott Smith, Trevor Hyde, Eduard Looijenga, Dan Margalit, Joshua Mundinger, Carlos A. Serv{\'a}n, Claire Voisin, and Nicholas Wawrykow for their helpful comments and discussions.

\section{Homomorphisms $B_n\to\PSL_2\Z$ and $B_n\to B_3$}\label{sec:homfromBn}

We begin by establishing some relevant facts about the braid groups. Artin's presentation of $B_n$ is
\begin{align*}
  \label{eq:Bn-presentation}
  \begin{split}
    B_n = \langle \sigma_1,\ldots,\sigma_{n-1} \mid{}& \sigma_i\sigma_{i+1}\sigma_i = \sigma_{i+1}\sigma_i\sigma_{i+1}, \quad 1\leq i < n-1 \\
                                                     & \sigma_i\sigma_j=\sigma_j\sigma_i, \quad 1\leq i<j-1<n-1 \rangle.
  \end{split}
\end{align*}

The holomorphic map $R\colon\Conf_4\C\to\Conf_3\C$ induces the homomorphism $R_*\colon B_4\twoheadrightarrow B_3$ given by $\sigma_1,\sigma_3\mapsto\sigma_1$ and $\sigma_2\mapsto\sigma_2$.

From Artin's presentation it can be seen that $B_3=\langle\alpha,\beta\mid\alpha^3=\beta^2\rangle$, where $\alpha=\sigma_1\sigma_2$ and $\beta=\sigma_1\sigma_2\sigma_1$. Hence $Z(B_3)=\langle\alpha^3\rangle=\langle\beta^2\rangle$, and there is a short exact sequence
\[
  1 \to Z(B_3) \to B_3 \to \PSL_2\Z \to 1.
\]

Bearing these facts in mind, we can state the two main theorems to be shown in this section.
\begin{theorem}\label{thm:BntoG}
  If $\varphi\colon B_n\to \PSL_2\Z$ is a homomorphism, then the following hold.
  \begin{enumerate}[label=(\roman*)]
  \item If $n\geq5$, then $\varphi$ is cyclic.
  \item If $n=4$, then $\varphi$ factors through $R_*\colon B_4\twoheadrightarrow B_3$.
  \item If $n=3$ and $\varphi$ is non-cyclic, then $\varphi$ descends to an endomorphism of $\PSL_2\Z$.
  \end{enumerate}
\end{theorem}

Since $\PSL_2\Z\cong\Z/2*\Z/3$, it is possible to classify its many endomorphisms, see \Cref{prop:End-PSL2Z}. This gives a classification of homomorphisms $B_n\to\PSL_2\Z$.

We also classify the homomorphisms $B_3\to\PSL_2\Z$ which send $\sigma_1$ and $\sigma_2$ to parabolic elements of $\PSL_2\Z$. This additional constraint will be relevant in \Cref{sec:conf-spaces} where we classify holomorphic maps.

We write $\begin{bsmallmatrix}a&b\\c&d\end{bsmallmatrix}$ for the image of $\begin{psmallmatrix}a&b\\c&d\end{psmallmatrix}\in\SL_2\Z$ in $\PSL_2\Z$.

\begin{theorem}
  \label{thm:b3-to-parabolic}
  If $\varphi\colon B_3\to\PSL_2\Z$ is a non-cyclic homomorphism with $\varphi(\sigma_1)$ and $\varphi(\sigma_2)$ parabolic, then it is given by composing the quotient map $B_3\twoheadrightarrow B_3/Z(B_3)$, with some isomorphism $B_3/Z(B_3)\cong\PSL_2\Z$. In other words, up to automorphisms of $\PSL_2\Z$ it is given by
  \[
    \varphi(\sigma_1) = \begin{bmatrix} 1 & 1 \\ 0 & 1 \end{bmatrix}, \qquad \varphi(\sigma_2) = \begin{bmatrix} 1 & 0 \\ -1 & 1 \end{bmatrix}.
  \]
\end{theorem}

Note that the quotient map $B_3\twoheadrightarrow\PSL_2\Z$ lifts to a surjection $H_*\colon B_3\twoheadrightarrow\pi_1^{\orb}(\M_{1,1})\cong\SL_2\Z$ given by
\[
  \sigma_1 \mapsto
  \begin{pmatrix}
    1 & 1 \\
    0 & 1
  \end{pmatrix},
  \qquad
  \sigma_2 \mapsto
  \begin{pmatrix}
    1 & 0 \\
    -1 & 1
  \end{pmatrix}.
\]

In \Cref{sec:BntoB3} we will address the implications of \Cref{thm:BntoG,thm:b3-to-parabolic} for homomorphisms $B_n\to B_3$.

A homomorphism $\varphi\colon G\to H$ is \emph{cyclic} if the image $\varphi(G)$ is a cyclic subgroup of $H$. We will use the following extensively.

\begin{lemma}[{\cite[Corollary 2]{For96}}]
  \label{lem:cyclic-homomorphism}
  Let $\varphi\colon B_n\to G$ be a homomorphism. Then either of the following conditions imply that $\varphi$ is cyclic.
  \begin{enumerate}[label=(\roman*)]
  \item $\varphi(\sigma_i)$ commutes with $\varphi(\sigma_{i+1})$ for some $1\leq i<n-1$.
  \item $n\geq5$, and $\varphi(\sigma_i)=\varphi(\sigma_j)$ for some $1\leq i<j<n$.
  \end{enumerate}
\end{lemma}

\subsection{$B_3\to\PSL_2\C$ and $B_3\to\PSL_2\Z$}\label{sec:B3toPSL}

\begin{lemma}
  \label{lem:B3-to-PSL2C}
  If $\varphi\colon B_3\to\PSL_2\C$ is non-cyclic, then $Z(B_3)\leq\ker(\varphi)$. In particular $\varphi$ descends to a homomorphism $\PSL_2\Z\to\PSL_2\C$.
\end{lemma}

\begin{proof}
  Write $a=\varphi(\sigma_1\sigma_2)$ and $b=\varphi(\sigma_1\sigma_2\sigma_1)$, so that $a^3=b^2$. It suffices to show that either $a$ and $b$ commute, or $a^3=b^2=I$. Non-identity powers of a given M\"{o}bius transformation have the same fixed point set in $\CP^1$. Hence, if $a^3=b^2$ is not the identity, then $a$ and $b$ both have the same fixed point set, and thus $a$ and $b$ commute by \cite[Theorem 4.3.5(ii)]{Bea83}.
\end{proof}

\begin{remark}
  Since $H^2(B_3;\Z/2)=0$ \cite{Fuk70}, any homomorphism $B_3\to\PSL_2\C$ lifts to a representation $B_3\to\SL_2\C$. Formanek classified the two-dimensional irreducible representations of $B_3$ \cite[Theorem 11]{For96}. From that classification one can easily check that $Z(B_3)$ has image contained in $\{\pm I\}$ under any irreducible representation $B_3\to\SL_2\C$. However \Cref{lem:B3-to-PSL2C} also handles the reducible indecomposable representations $B_3\to\SL_2\C$.
\end{remark}

The following proposition classifies the endomorphisms of $\PSL_2\Z$.

\begin{proposition}
  \label{prop:End-PSL2Z}
  Let $a$, $b$ denote the images of $\alpha=\sigma_1\sigma_2$, $\beta=\sigma_1\sigma_2\sigma_1$ respectively under the quotient map $B_3\twoheadrightarrow\PSL_2\Z$, so that $\PSL_2\Z=\langle a\rangle*\langle b\rangle$ and $a^3=b^2=1$. Then the non-cyclic endomorphisms of $\PSL_2\Z$ are precisely those of the form
  \[ a \mapsto ga^{\pm1}g^{-1}, \quad b \mapsto hbh^{-1}, \]
  for some group elements $g,h\in\PSL_2\Z$.
\end{proposition}

\begin{proof}
  By the free product description of $\PSL_2\Z$ all maps above extend to endomorphisms. Since the order of $a$ and the order of $b$ are coprime integers, Kurosh's theorem \cite{Kur34} implies that any endomorphism must carry $\langle a\rangle$ into a conjugate subgroup $g\langle a\rangle g^{-1}$, and also $\langle b\rangle$ to a conjugate subgroup $h\langle b\rangle h^{-1}$. Hence these are the only possibilities.
\end{proof}

We now describe the non-cyclic homomorphisms $B_3\to\PSL_2\Z$ where the images of the standard generators are parabolic.

\begin{proof}[Proof of \Cref{thm:b3-to-parabolic}]
Without loss of generality, we may assume that $\varphi(\sigma_1)=\begin{bsmallmatrix}1&x\\0&1\end{bsmallmatrix}$ for some integer $x>0$. Let $\eta\in\PP^1(\Q)=\Q\cup\{\infty\}$ be the unique fixed point of $\varphi(\sigma_2)$. Since $\varphi$ is non-cyclic, $\eta\neq\infty$.

  Let $T_\eta\coloneqq\begin{bsmallmatrix}1&\eta\\0&1\end{bsmallmatrix}$, and define $\psi\colon B_3\to\PSL_2\Q$ by $\psi(g)=T_\eta^{-1}\cdot\varphi(g)\cdot T_\eta$. Note that $\psi(\sigma_2)$ fixes $0\in\PP^1(\Q)$, hence its top right entry is zero. We can directly compute the remaining entries in terms of $\varphi(\sigma_2)=\begin{bsmallmatrix} a & \ast \\ c & d \end{bsmallmatrix}$ to be
  \[
    \psi(\sigma_2) = \begin{bmatrix} a - c\eta & 0 \\ c & d + c\eta \end{bmatrix}.
  \]
  Since $\det(\psi(\sigma_2))=1\in\Z$ and $a,d\in\Z$, we must have $c\eta\in\Z$. Hence $\psi(\sigma_2)=\begin{bsmallmatrix}1&0\\y&1\end{bsmallmatrix}$, where $y=\pm c\in\Z$ is non-zero. The braid relation implies
  \begin{gather*}
    \begin{bmatrix} 1 & x \\ 0 & 1 \end{bmatrix} \begin{bmatrix} 1 & 0 \\ y & 1 \end{bmatrix} \begin{bmatrix} 1 & x \\ 0 & 1 \end{bmatrix} = \begin{bmatrix} 1 & 0 \\ y & 1 \end{bmatrix} \begin{bmatrix} 1 & x \\ 0 & 1 \end{bmatrix} \begin{bmatrix} 1 & 0 \\ y & 1 \end{bmatrix} \\[0.5em] \implies  xy + 1 = 0 \implies (x,y) = (1,-1).
  \end{gather*}
  Hence $c=\pm1$, and thus $\eta\in\Z$. Thus $T_\eta\in\PSL_2\Z$, and $\psi\colon B_3\to\PSL_2\Z$ has the required form, completing the proof.
\end{proof}

\subsection{$B_n\to\PSL_2\Z$}\label{sec:BntoPSL}

We first prove a useful lemma about $\PSL_2\Z$, by viewing it as a subgroup of $\Isom^+(\HH)\cong\PSL_2\R$.

\begin{lemma}\label{lem:comm-conj}
If two elements of $\PSL_2\Z$ commute and are conjugate, then they must be either equal or inverses.
\end{lemma}
\begin{proof}
  If both are hyperbolic they have the same translation axis and translation length. If they are both elliptic they have a common fixed point and the same unsigned angle of rotation. If they are both parabolic we may assume they both fix $\infty\in\PP^1(\Q)$. Hence both are powers of $\begin{bsmallmatrix} 1 & 1 \\ 0 &1 \end{bsmallmatrix}$, thus equal since they are conjugate in $\PSL_2\Z$. In all three cases they must be equal or inverses.
\end{proof}

\begin{proof}[Proof of \Cref{thm:BntoG}]
  \Cref{lem:B3-to-PSL2C} implies (iii). It follows from \Cref{lem:cyclic-homomorphism}(ii) that (ii)$\implies$(i). Suppose then that $n=4$.
By \Cref{lem:comm-conj}, $\varphi(\sigma_1)$ and $\varphi(\sigma_3)$ are either equal or inverses. By \Cref{prop:End-PSL2Z} if $\varphi$ is non-cyclic, then $\varphi(\sigma_i)$ and $\varphi(\sigma_{i+1})$ are mapped to the same generator of the abelianization $\Z/6\Z$ of $\PSL_2\Z$, for $i=1,2$. This rules out the possibility that $\varphi(\sigma_1)$ and $\varphi(\sigma_3)$ are inverses. Hence $\varphi(\sigma_1)=\varphi(\sigma_3)$, which proves (ii).
\end{proof}

\subsection{Lifts to $B_n\to B_3$}\label{sec:BntoB3}

Throughout this paper we will denote the centralizer of a group $H$ in a group $G$ by $C_G(H)$, and $C_G(g)\coloneqq C_G(\langle g\rangle)$ for $g\in G$.

\begin{definition}[\bf{Transvection}]
  Let $\varphi\colon G\to H$ be a homomorphism. If $t\colon G\to C_H(\varphi(G))$ is a cyclic homomorphism, then the \emph{transvection of $\varphi$ by $t$} is the homomorphism $\varphi^t\colon G\to H$ defined by
  \[
    \varphi^t(g) = \varphi(g)t(g).
  \]
\end{definition}

\begin{theorem}
  \label{thm:bntob3}
  Let $\varphi\colon B_n\to B_3$ be a homomorphism.
  \begin{enumerate}[label=(\roman*)]
  \item If $n\geq5$, then $\varphi$ is cyclic.
  \item If $n=4$, then $\varphi$ factors through $R_*$.
  \item If $n=3$ and $\varphi$ is non-cyclic, then up to a transvection by a homomorphism $B_3\to Z(B_3)$ and post-composing by an automorphism of $B_3$, the map $\varphi$ has the form
    \[
      \alpha \mapsto g\alpha g^{-1}, \quad \beta \mapsto \beta,
    \]
    for some $g\in B_3$.
  \end{enumerate}
\end{theorem}

\begin{proof}
  If we post-compose the above homomorphisms with the quotient map $B_3\twoheadrightarrow B_3/Z(B_3)\cong\PSL_2\Z$, then we obtain all homomorphisms classified in \Cref{thm:BntoG}. Furthermore, any two homomorphisms $B_n\to B_3$ that lift the same homomorphism $B_n\to\PSL_2\Z$ can only differ by a transvection $B_n\to Z(B_3)$. Hence this is a complete list of homomorphisms $B_n\to B_3$.
\end{proof}

Dyer and Grossman \cite{DG81} have shown that $\Aut(B_n)=\Inn(B_n)\rtimes\langle\iota\rangle$ for $n\geq2$, where $\iota\colon B_n\to B_n$ is the inversion automorphism given by $\sigma_i\mapsto\sigma_i^{-1}$ for $i=1,\ldots,n-1$, so this gives a complete classification of the homomorphisms $B_n\to B_3$.

In \Cref{thm:bntob3} the case $n\geq5$ is a special case of a Theorem of Lin \cite[Theorem 3.1]{Lin04b}, and $n=4$ is already known by Chen, Kordek and Margalit \cite[Proposition 8.6]{CKM23}. The case $n=3$ is stated by Orevkov \cite[Proposition 1.8]{Ore22} without proof. Our proof of Theorem~\ref{thm:BntoG} is independent of these results.

\section{Nielsen--Thurston theory}\label{sec:Nielsen-Thurston}

The purpose of this section is twofold: to recall well-known facts about the braid group from the Nielsen--Thurston point of view, while also establishing the notation and conventions we use.

There is a canonical isomorphism between the braid group $B_n$ and the mapping class group $\Mod(\D_n)$ of $\D_n$. Here $\D_n$ is a disk including its boundary with $n$ unordered marked points in the interior. Recall that the mapping class group of a surface $\Mod(S)$ is the group of orientation preserving homeomorphisms of $S$ that preserve its marked points as a set, and fix each boundary component pointwise.

The Nielsen--Thurston classification theorem states that every element of a mapping class group, hence also every braid, is either reducible, periodic, or pseudo-Anosov \cite{Thu88}. We will outline what this means for the braid group $B_n$. For more detail, see \cite[Chapters 9 and 13]{FM12}.

Our strategy for classifying homomorphisms $B_n\to B_4$ involves considering the possible Nielsen--Thurston types of the image of a generator of $Z(B_n)$ in $B_4$, and its possible canonical reduction systems.

\subsection{Canonical reduction systems}

A \emph{simple closed curve} in $\D_n$ is an isotopy class of smooth maps $S^1\to\D_n$ without self intersection that avoid the $n$ marked points. For convenience, we do not consider null-homotopic curves to count as simple closed curves. The braid group $B_n$ acts on the set of simple closed curves in $\D_n$.

A \emph{multicurve} is a set of disjoint simple closed curves. For our purposes, it will be convenient to allow multicurves to contain non-essential simple closed curves, i.e., curves that are isotopic to a marked point or boundary component. We say a multicurve is \emph{nontrivial} if it contains at least one essential component, and we say it is \emph{essential} if all its components are essential.

A \emph{reduction system} of a braid $g\in B_n$ is an essential multicurve preserved by $g$. The \emph{canonical reduction system} $\crs(g)$ of $g$ is the intersection of all maximal reduction systems of $g$. In the special case of the surface $\D_n$ there is one additional notion we may define: the \emph{outer canonical reduction system} $\ocrs(g)$. We define this to be the outermost loops in the canonical reduction system, together with a small loop surrounding every marked point not already encircled by a component of the canonical reduction system.

Note that $\crs(ghg^{-1})=g\cdot\crs(h)$ and $\ocrs(ghg^{-1})=g\cdot\ocrs(h)$. In particular commuting elements preserve each other's (outer) canonical reduction systems.

The outer canonical reduction systems in $\D_n$, up to the action of $B_n$, are in bijection with the partitions of $n$: the partition of $\ocrs(g)$ is obtained by recording the number of marked points each component of $g$ encircles.

\subsection{Reducible braids}

A braid is \emph{reducible} if its canonical reduction system is non-empty. This is equivalent to its outer canonical reduction system being non-trivial.

Important examples of reducible braids include \emph{Dehn twists} $T_\gamma$ about a simple closed curve $\gamma$, see \cite[Chapter 3]{FM12} for a definition. The Dehn twist $z$ about a curve isotopic to the boundary component of $\D_n$ generates $Z(B_n)$. A \emph{multitwist} is a product $\prod_{i=1}^kT_{\gamma_i}^{p_i}$ of Dehn twists, where $p_i\in\Z$ and the curves $\gamma_i$ are pairwise disjoint. The canonical reduction system of such a product is the set of all $\gamma_i$ that are not isotopic to a boundary component.

If $\varphi\colon G\to B_m$ is a homomorphism, we say it is \emph{reducible} if some essential multicurve is preserved by $\varphi(G)$. If $G=B_n$ and $z$ is a generator of $Z(B_n)$, then $\varphi$ is reducible if and only if $\varphi(z)$ is. For this reason we define $\crs(\varphi)\coloneqq\crs(\varphi(z))$ and $\ocrs(\varphi)\coloneqq\ocrs(\varphi(z))$.

We will need to work with the stabilizer $B_\Gamma$ in $B_n$ of an outer canonical reduction system $\Gamma\coloneqq\{\gamma_1,\ldots,\gamma_m\}$ in $\D_n$. If we collapse the disks bounded by each $\gamma_i$ to a point, then we obtain a disk $\D_m$ whose $i$th marked point can be labelled with the number $n_i$ of marked points originally encircled by $\gamma_i$.

Given a labelling $L$ of the marked points in $\D_m$, let $B_L$ be the subgroup of $B_m$ preserving this labelling. Groups of the form $B_L$ are \emph{mixed braid groups}. The labelling $L$ determines a partition of the $m$ marked points. If $k_1,\ldots,k_\ell$ are the sizes of the parts of this partition, then these integers determine the isomorphism type of $B_L$, and we write $B_L=B_{k_1,\ldots,k_\ell}$.

If a labelling is obtained by the above construction collapsing the disks bounded by components of $\Gamma$, then denote the labelling by $[\Gamma]$. See \Cref{fig:split} for an example. For example, if $\Gamma\coloneqq\{\gamma_1,\ldots,\gamma_5\}$ is the outer canonical reduction system in \Cref{fig:split}, then $B_{[\Gamma]}\cong B_{1,2,1,1}<B_5$.

\begin{figure}
  \centering
  \begin{tikzpicture}
    \draw (0.15,0) circle (3);
    \node at (0.4,0)[circle,fill,inner sep=1pt]{};
    \node at (0.8,0)[circle,fill,inner sep=1pt]{};
    \node at (1.2,0)[circle,fill,inner sep=1pt]{};
    \node at (1.6,0)[circle,fill,inner sep=1pt]{};
    \node at (2,0)[circle,fill,inner sep=1pt]{};
    \node at (2.4,0)[circle,fill,inner sep=1pt]{};
    \node at (0,0)[circle,fill,inner sep=1pt]{};
    \node at (-0.4,0)[circle,fill,inner sep=1pt]{};
    \node at (-0.8,0)[circle,fill,inner sep=1pt]{};
    \node at (-1.2,0)[circle,fill,inner sep=1pt]{};
    \node at (-1.6,0)[circle,fill,inner sep=1pt]{};
    \node at (-2,0)[circle,fill,inner sep=1pt]{};
    \node at (-2.4,0)[circle,fill,inner sep=1pt]{};
    \node at (-2.5,2.5){$\D_{13}$};
    \draw (-1.8,0) ellipse (0.75 and 0.4);
    \node at (-1.8,-0.75) {$\gamma_1$};
    \draw (-0.4,0) ellipse (0.55 and 0.35);
    \node at (-0.4,-0.75) {$\gamma_2$};
    \draw (0.8,0) ellipse (0.55 and 0.35);
    \node at (0.8,-0.75) {$\gamma_3$};
    \draw (1.8,0) ellipse (0.35 and 0.3);
    \node at (1.8,-0.75) {$\gamma_4$};
    \draw (2.4,0) circle (0.2);
    \node at (2.4,-0.75) {$\gamma_5$};
    \draw (6,0) circle (2);
    \node at (4.2,1.8){$\D_5$};
    \node at (6,0)[circle,fill,inner sep=1pt]{};
    \node at (4.8,-0.4){$4$};
    \node at (6.6,0)[circle,fill,inner sep=1pt]{};
    \node at (5.4,-0.4){$3$};
    \node at (7.2,0)[circle,fill,inner sep=1pt]{};
    \node at (6,-0.4){$3$};
    \node at (5.4,0)[circle,fill,inner sep=1pt]{};
    \node at (6.6,-0.4){$2$};
    \node at (4.8,0)[circle,fill,inner sep=1pt]{};
    \node at (7.2,-0.4){$1$};
  \end{tikzpicture}
  \caption{Labelling from an outer canonical reduction system.}
  \label{fig:split}
\end{figure}
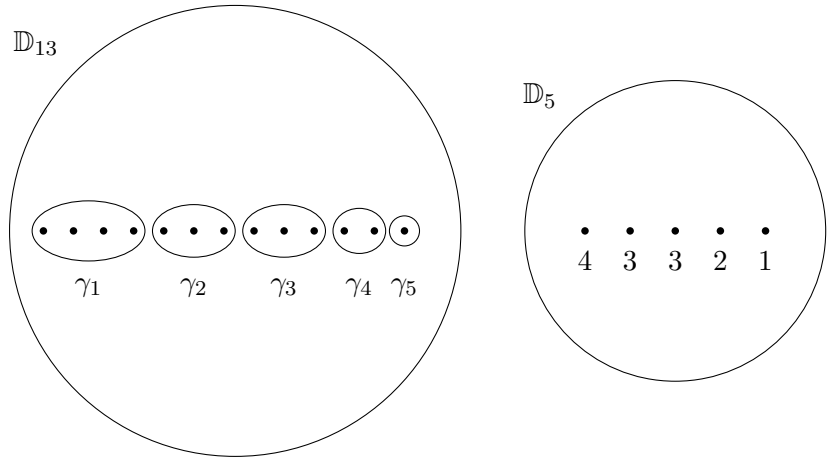

\begin{lemma}[{\cite[Lemma 3.1]{KM22}}]
  \label{lem:split}
  The short exact sequence
  \[
    1\to (B_{n_1}\times\cdots\times B_{n_m})\to B_\Gamma\to B_{[\Gamma]}\to1
  \]
  obtained by collapsing the disks bounded by $\gamma_i$ to a point, has a section. We obtain a semi-direct product $B_\Gamma \cong (B_{n_1}\times\cdots\times B_{n_m})\rtimes B_{[\Gamma]}$. The action of $B_{[\Gamma]}$ on $B_{n_1}\times\cdots\times B_{n_m}$ factors through the homomorphism $B_{[\Gamma]}\to S_m$ and is given by permuting the factors.
\end{lemma}

\subsection{Periodic braids}

A braid in $B_n$ is \emph{periodic} if it has finite order in $B_n/Z(B_n)$.

Let
  \[
    \alpha_n\coloneqq\sigma_1\cdots\sigma_{n-1}, \quad \beta_n\coloneqq\alpha_n\sigma_1, \quad z_n\coloneqq\alpha_n^n=\beta_n^{n-1}.
  \]

   It is a result of Chow that $Z(B_n)$ is infinite cyclic generated by $z_n$ \cite{Cho48}. We will continue to write $\alpha=\alpha_3$ and $\beta=\beta_3$.

The following proposition follows from work of Ker\'ekj\'art\'o and Eilenberg.

\begin{proposition}[\cite{Ker19,Eil34}]
  \label{prop:what-is-periodic}
  The periodic elements of $B_n$ are all conjugate to a power of either $\alpha_n$ or $\beta_n$.
\end{proposition}

Gonz\'alez-Meneses and Wiest attribute the computation of centralizers of non-central periodic braids below to Bessis, Digne and Michel \cite{BDM02}.

\begin{proposition}\cite[Theorems 3.2 and 3.4]{GW04}
  \label{prop:periodic}\
  \begin{enumerate}[label=(\alph*)]
  \item Let $d=\gcd(k,n)$, and suppose that $d<n$. The centralizer of $\alpha_n^k$ in $B_n$ is isomorphic to the mixed braid group $B_{d,1}$.
  \item Let $d=\gcd(k,n-1)$, and suppose that $d<n-1$. The centralizer of $\beta_n^k$ in $B_n$ is isomorphic to $B_{d,1}$.
  \end{enumerate}
\end{proposition}

\subsection{Pseudo-Anosov braids}

We will not define pseudo-Anosov braids here, we refer to \cite[Chapter 13]{FM12} for a definition. By the Nielsen--Thurston classification theorem, they are precisely the non-periodic irreducible braids.

The centralizers of pseudo-Anosov braids are well understood:

\begin{proposition}[{\cite[Proposition 4.1]{GW04}}]\label{prop:central}
  The centralizer in $B_n$ of a pseudo-Anosov braid is isomorphic to $\Z^2$.
\end{proposition}

By Lemma \ref{lem:cyclic-homomorphism} this implies that if $\rho\colon B_n\to B_m$ is a homomorphism and $\rho(z_n)$ is pseudo-Anosov, then $\rho$ has cyclic image.

\section{Equivalence classes of homomorphisms}\label{sec:equivalence-classes}

Following \cite{CKM23}, we define an equivalence relation on the set of homomorphisms $G\to H$ as follows:

\begin{definition}[\textbf{Equivalence of homomorphisms}]
  Two homomorphisms $G\to H$ are \emph{equivalent} if they are related by post-composition by an automorphism of $H$ and a sequence of transvections.
\end{definition}

The goal of this section is to prove \Cref{thm:equivalence-classes}, namely that there are infinitely many equivalence classes of homomorphisms $B_n\to B_m$ when $n\in\{3,4\}$ and $m\geq3$.

We will also give an explicit list of the equivalence classes of homomorphisms $B_n\to B_3$. 
To do this, we first compute the centralizers of the non-cyclic homomorphisms from \Cref{thm:bntob3}.

\subsection{Centralizers}

We will repeatedly use the following result of Gonz\'alez-Meneses and Wiest.
\begin{lemma}[{\cite[Propositions 3.3 and 3.5]{GW04}}]\label{GWlem}\
\begin{enumerate}[label=(\roman*)]
\item The centralizer of $\alpha_n$ in $B_n$ is $\langle\alpha_n\rangle$.
\item The centralizer of $\beta_n$ in $B_n$ is $\langle\beta_n\rangle$. 
\end{enumerate}
\end{lemma}

\begin{lemma}
  \label{lem:centralizer-b3tob3}
  If $\rho$ is a non-cyclic endomorphism of $B_3$, then the centralizer of $\rho(B_3)$ in $B_3$ is the center $Z(B_3)$.
\end{lemma}

\begin{proof}
  By \Cref{thm:bntob3} it suffices to show that $C_{B_3}(g\alpha g^{-1})\cap C_{B_3}(\beta)=Z(B_3)$ for all $g\in G$. The image of $g\alpha g^{-1}$ in $S_3$ is a 3-cycle, and the image of $\beta$ is a 2-cycle. Therefore, the intersection of $C_{B_3}(g\alpha g^{-1})=\langle g\alpha g^{-1}\rangle$ with $C_{B_3}(\beta)=\langle\beta\rangle$ is contained in, and hence equal to, $\langle (g\alpha g^{-1})^3 \rangle = \langle\beta^2 \rangle = Z(B_3)$.
\end{proof}

\subsection{Reduction to endomorphisms of $B_3$} Next we show that Castel's conjecture, the $n=m=3$ case of \Cref{thm:equivalence-classes}, implies all other cases of \Cref{thm:equivalence-classes}.

\begin{lemma}
  \label{lem:no-collapsing}
  Inequivalent endomorphisms of $B_3$ remain inequivalent as homomorphisms after either one or both of the following operations
  \begin{enumerate}[label=(\roman*)]
  \item Post-composition with the standard inclusion $B_3\hookrightarrow B_m$
  \item Pre-composition with $R_*\colon B_4\twoheadrightarrow B_3$.
  \end{enumerate}
\end{lemma}

\begin{proof}
  \begin{figure}
    \centering
    \begin{tikzpicture}
      \draw (0,0) circle (2);
      \node at (-1.4,0)[circle,fill,inner sep=1pt]{};
      \node at (-1,0)[circle,fill,inner sep=1pt]{};
      \node at (-0.6,0)[circle,fill,inner sep=1pt]{};
      \node at (0,0)[circle,fill,inner sep=1pt]{};
      \node at (0.7,0){$\cdots$};
      \node at (1.4,0)[circle,fill,inner sep=1pt]{};
      \draw (-1,0) ellipse (0.6 and 0.4);
      \node at (-1,-0.75) {$\gamma$};
      \draw [decorate,decoration={brace,amplitude=5pt,mirror,raise=1ex}]
      (-0.1,0) -- (1.5,0) node[midway,yshift=-1.75em]{\small $m-3$};
    \end{tikzpicture}
    \caption{}
    \label{fig:B3-in-Bm}
  \end{figure}
  Since $R_*$ is surjective, (ii) doesn't collapse equivalence classes of homomorphisms. So it suffices to show (i) doesn't collapse equivalence classes of endomorphisms of $B_3$.

  Suppose that $\varphi\colon B_3\to B_m$ is a non-cyclic endomorphism of $B_3$ postcomposed with the standard inclusion $B_3\hookrightarrow B_m$. Let $t\colon B_3\to B_{1,m-3}$ be a cyclic homomorphism. View $B_3\times B_{1,m-3}$ as a subgroup of $B_m$, namely the stabilizer of the simple closed curve $\gamma$ in \Cref{fig:B3-in-Bm} surrounding three marked points, cf. \Cref{lem:split}. We claim $C_{B_m}(\varphi^t(B_3))\leq Z(B_3)\times B_{1,m-3}$.

  First note that $\varphi$ and $\varphi^t$ agree on $[B_3,B_3]$. By the classification of endomorphisms of $B_3$, $\varphi(B_3)\leq B_3$ surjects onto $S_3$, hence $\varphi^t([B_3,B_3])=\varphi([B_3,B_3])$ surjects onto $A_3$. If $g\in\varphi^t([B_3,B_3])$ has non-trivial image in $A_3$, then its canonical reduction system in $\D_m$ is $\{\gamma\}$. By a result of Gonz\'alez-Meneses and Wiest \cite[Theorem 1.1]{GW04}, $C_{B_m}(g)=C_{B_3}(g)\times B_{1,m-3}$. In particular $C_{B_m}(\varphi^t(B_3))\leq B_3\times B_{1,m-3}$. Now \Cref{lem:centralizer-b3tob3} implies $C_{B_m}(\varphi^t(B_3))\leq Z(B_3)\times B_{1,m-3}$.

  This implies by induction that repeated transvections of $\varphi$ can be realized by one transvection in $B_3$ together with one transvection by a cyclic homomorphism $t\colon B_3\to B_{1,m-3}\leq B_m$. Suppose that $\tau\in\Aut(B_3)$, $t\colon B_3\to B_{1,m-3}$ is a cyclic homomorphism, and that $\tau(\varphi^t(B_3))\leq B_3\leq B_m$. Since $\varphi^t([B_3,B_3])=\varphi([B_3,B_3])\leq B_3$ surjects onto $A_3$, it follows that $\tau$ must preserve $\gamma$. Thus $\tau$ restricts to an automorphism of $B_3\leq B_m$. This implies $t$ is trivial.

  Hence, if two homomorphisms $\varphi_1,\varphi_2\colon B_3\to B_3\hookrightarrow B_m$ are equivalent as homomorphisms $B_3\to B_m$, then they are also equivalent as endomorphisms of $B_3$.
\end{proof}

\subsection{Equivalence classes of endomorphisms of $B_3$} Let $\rho_g$ be the endomorphism of $B_3$ defined by $\alpha\mapsto g\alpha g^{-1}$, $\beta\mapsto\beta$. By \Cref{thm:bntob3} every endomorphism of $B_3$ is equivalent to some $\rho_g$, hence \Cref{thm:equivalence-classes} is equivalent to the $\rho_g$ representing infinitely many equivalence classes. The next lemma shows that these equivalence classes are in bijective correspondence with certain double cosets in $B_3$.

\begin{lemma}\label{lem:bijective-double}
  The endomorphisms $\rho_g$ and $\rho_{g'}$ are equivalent if and only if the double cosets $\langle\beta\rangle g\langle\alpha\rangle$ and $\langle\beta\rangle g'\langle\alpha\rangle$ are equal.
\end{lemma}
\begin{proof}
  First note that $\rho_g$ itself only depends on the coset $g\langle\alpha\rangle$. Conjugating by $\beta$ shows that its equivalence class only depends on the double coset $\langle\beta\rangle g\langle\alpha\rangle$.

  Conversely, suppose that $\rho_g$ is equivalent to $\rho_{g'}$. By \Cref{lem:centralizer-b3tob3}, there exists a homomorphism $t\colon B_3\to Z(B_3)$ and $\tau\in\Aut(B_3)$ such that
  \begin{equation}
    \label{eq:equivalent}
    \rho_{g'} = \tau \circ \rho_g^t.
  \end{equation}
  By restricting both sides of \eqref{eq:equivalent} to $Z(B_3)$, we can conclude that $t$ is trivial and $\tau\in\Inn(B_3)$. By evaluating both sides of \eqref{eq:equivalent} at $\beta$, we see that $\tau$ fixes $\beta$. Hence $\tau\colon g\mapsto hgh^{-1}$ for some $h\in C_{B_3}(\beta)=\langle\beta\rangle$. By evaluating both sides of \eqref{eq:equivalent} at $\alpha$ we find that $(g')^{-1}hg\in C_{B_3}(\alpha)=\langle \alpha \rangle$. Hence $\langle\beta\rangle g\langle\alpha\rangle=\langle\beta\rangle g'\langle\alpha\rangle$.
\end{proof}

Everything is now in place to prove Theorem \ref{thm:equivalence-classes}.

\begin{proof}[Proof of Theorem \ref{thm:equivalence-classes}]
  By \Cref{lem:no-collapsing,lem:bijective-double} it suffices to show that there are infinitely many double cosets $\langle\beta\rangle g\langle\alpha\rangle$. Recall that $B_3/Z(B_3)\cong\langle a\rangle*\langle b\rangle\cong\Z/3*\Z/2$, where $a$ and $b$ are the images of $\alpha$ and $\beta$ respectively. By considering the usual normal form for elements of a free product, the double cosets $\langle b\rangle (ab)^n\langle a\rangle=\{b^i(ab)^na^j : i,j\in\Z\}$ are distinct as $n$ ranges over the positive integers. Therefore, the double cosets $\langle\beta\rangle(\alpha\beta)^n\langle\alpha\rangle$ are also all distinct as $n$ ranges over the positive integers.
\end{proof}

\section{Classification of homomorphisms $B_3\to B_4$ and $B_4\to B_4$}\label{sec:B3toB3andB4}

A major goal of this section is to prove the following theorem.

\begin{theorem}
  \label{thm:B3-to-B4}
  The non-cyclic homomorphisms $B_3\to B_4$ are, up to a transvection $B_3\to Z(B_4)$ and post-composing by an automorphism of $B_4$, one of the following.
  \begin{enumerate}[label=(\roman*)]
  \item An endomorphism of $B_3$, followed by the standard inclusion $B_3\hookrightarrow B_4$.
  \item Given by $\alpha\mapsto g\beta_4g^{-1}$, $\beta\mapsto\alpha_4^2$ for some $g\in B_4$.
  \end{enumerate}
\end{theorem}

We provide a proof of \Cref{thm:B3-to-B4} based on the Nielsen--Thurston classification theorem and canonical reduction systems. These methods were suggested to us by Farb. \Cref{thm:B3-to-B4} is also stated by Orevkov \cite[Proposition 1.9]{Ore22}. Although a proof is not provided, the methods suggested there are the same as our own.

We also prove the following theorem in \Cref{sec:pseudo-periodic}, which classifies the homomorphisms $B_3\to B_4$ satisfying certain conditions. In \Cref{sec:conf-spaces} these will be shown to be necessary conditions to be induced by a non-root holomorphic map.

\begin{theorem}
  \label{thm:b3-to-b4-pseudoperiodic}
  The homomorphism $\Psi_{3*}\colon B_3\to B_4$ given by
  \[
    \Psi_{3*}(\sigma_1) = \sigma_1\sigma_2, \qquad \Psi_{3*}(\sigma_2) = \sigma_3\sigma_2,
  \]
  is, up to automorphisms of $B_4$ and transvections $B_3\to Z(B_4)$, the only non-cyclic, irreducible homomorphism $\varphi\colon B_3\to B_4$ such that $\varphi(\sigma_1)$ has some positive power that is a multitwist.
\end{theorem}

A complete proof of \Cref{thm:b3-to-b4-pseudoperiodic} will rely on the solution to an equation in the free group of rank 2 that \Cref{sec:f2} is dedicated to proving.

Chen, Kordek, and Margalit prove that every endomorphism of $B_4$ is either a transvection of the identity, or factors through the homomorphism $R_*\colon B_4\twoheadrightarrow B_3$ \cite[Theorem 8.4]{CKM23}. Together with \Cref{thm:B3-to-B4}, this implies the following classification of endomorphisms of $B_4$.

\begin{theorem}
  \label{thm:B4-to-B4}
  The non-cyclic endomorphisms of $B_4$ are, up to a transvection $B_4\to Z(B_4)$ and post-composing by an automorphism of $B_4$, one of the three following types.
  \begin{enumerate}[label=(\roman*)]
  \item The identity.
  \item The resolving quartic homomorphism $R_*\colon B_4\twoheadrightarrow B_3$, followed by an endomorphism of $B_3$, followed by the standard inclusion $B_3\hookrightarrow B_4$.
  \item The resolving quartic homomorphism $R_*\colon B_4\twoheadrightarrow B_3$, followed by a homomorphism $B_3\to B_4$ given by $\alpha\mapsto g\beta_4g^{-1}$, $\beta\mapsto\alpha_4^2$ for some $g\in B_4$.
  \end{enumerate}
\end{theorem}

Let $\varphi\colon B_3\to B_4$ be a non-cyclic homomorphism, and let $Z(B_3)=\langle z_3\rangle$. \Cref{thm:B3-to-B4} follows from two lemmas that we prove in \Cref{sec:reducible,sec:other-cases}.

\subsection{The reducible case}\label{sec:reducible}

\begin{lemma}\label{lem:b3-to-b4-reducible}
  If $\varphi$ is reducible, then $\varphi$ belongs to case (i) of \Cref{thm:B3-to-B4}.
\end{lemma}

\begin{proof}[Proof of \Cref{lem:b3-to-b4-reducible}]

Suppose that $\varphi$ is reducible. Up to the action of $B_n$ there are three possible values of $\ocrs(\varphi)$. These are depicted in \Cref{fig:4-fix-curve,fig:4-fix-2-curve,fig:4-fix-3-curve}, and they correspond to the three non-trivial partitions of 4: $2+2$, $2+1+1$, and $3+1$. We consider these three cases in turn.

\begin{figure}
  \centering
  \begin{minipage}[c]{0.32\linewidth}
    \centering
    \begin{tikzpicture}
    \draw (0,0) circle (2);
    \node at (-1.2,0)[circle,fill,inner sep=1pt]{};
    \node at (-0.4,0)[circle,fill,inner sep=1pt]{};
    \node at (0.4,0)[circle,fill,inner sep=1pt]{};
    \node at (1.2,0)[circle,fill,inner sep=1pt]{};
    \draw (-0.8,0) ellipse (0.7 and 0.4);
    \node at (-0.8,-0.75) {$\gamma_1$};
    \draw (0.8,0) ellipse (0.7 and 0.4);
    \node at (0.8,-0.75) {$\gamma_2$};
  \end{tikzpicture}
  \caption{}
  \label{fig:4-fix-curve}
  \end{minipage}
  \begin{minipage}[c]{0.32\linewidth}
    \centering
\begin{tikzpicture}
\draw (0,0) circle (2);
\node at (-1.2,0)[circle,fill,inner sep=1pt]{};
\node at (-0.4,0)[circle,fill,inner sep=1pt]{};
\node at (0.4,0)[circle,fill,inner sep=1pt]{};
\node at (1.2,0)[circle,fill,inner sep=1pt]{};
\draw (-0.8,0) ellipse (0.7 and 0.4);
\node at (-0.8,-0.75) {$\gamma_1$};
\draw (0.4,0) circle (0.3);
\node at (0.4,-0.75) {$\gamma_2$};
\draw (1.2,0) circle (0.3);
\node at (1.2,-0.75) {$\gamma_3$};
\end{tikzpicture}
    \caption{}
    \label{fig:4-fix-2-curve}
  \end{minipage}
  \begin{minipage}[c]{0.32\linewidth}
    \centering
\begin{tikzpicture}
\draw (0,0) circle (2);
\node at (-1.2,0)[circle,fill,inner sep=1pt]{};
\node at (-0.4,0)[circle,fill,inner sep=1pt]{};
\node at (0.4,0)[circle,fill,inner sep=1pt]{};
\node at (1.2,0)[circle,fill,inner sep=1pt]{};
\draw (-0.4,0) ellipse (1.1 and 0.45);
\node at (-0.4,-0.75) {$\gamma_1$};
\draw (1.2,0) circle (0.3);
\node at (1.2,-0.75) {$\gamma_2$};
\end{tikzpicture}
    \caption{}
    \label{fig:4-fix-3-curve}
  \end{minipage}
  \captionsetup{labelformat=empty}
  \caption{Nontrivial outer canonical reduction systems in $B_4$.}
\end{figure}

\subsubsection{The partition $2+2$}\ \medskip

\textbf{Claim:} If $\ocrs(\varphi)$ is $\Gamma\coloneqq\{\gamma_1,\gamma_2\}$ as depicted in \Cref{fig:4-fix-curve} up to the action of $B_4$, then $\varphi$ is cyclic.

\begin{proof}
  By \Cref{lem:split}, $B_\Gamma\cong(B_2\times B_2)\rtimes B_2\cong\Z^2\rtimes\Z$, where $n\in\Z$ acts on $\Z^2$ by swapping coordinates if $n$ is odd. We present the group $\Z^2\rtimes\Z$ as $\langle x,y,z \mid [x,y]=1,\ zx=yz,\ zy=xz\rangle$, where every element of this group can be written uniquely in the form $x^ay^bz^c$. Write
  \[
    \varphi(\sigma_i) = x^{a_i}y^{b_i}z^{c_i}, \quad i=1,2.
  \]
Since $\langle x,y,z\rangle/\langle x,y\rangle\cong\langle z\rangle$ and since $\sigma_1$ and $\sigma_2$ are conjugate, we see that $c_1=c_2$. If this common value $c$ is even, then $\varphi$ is cyclic since $\langle x,y,z^2\rangle$ is abelian. So suppose that $c$ is odd. Since $\sigma_1\sigma_2\sigma_1=\sigma_2\sigma_1\sigma_2$, we find
  \[
    x^{2a_1+b_2}y^{2b_1+a_2}z^{3c} = x^{2a_2+b_1}y^{2b_2+a_1}z^{3c}.
  \]
  Hence $a_1=a_2$ and $b_1=b_2$, and so again $\varphi$ is cyclic. 
\end{proof}

\subsubsection{The partition $2+1+1$}\ \medskip

\textbf{Claim:} If $\ocrs(\varphi)$ is $\Gamma\coloneqq\{\gamma_1,\gamma_2,\gamma_3\}$ as depicted in \Cref{fig:4-fix-2-curve} up to the action of $B_4$, then $\varphi$ is cyclic.

\begin{proof}
  By \Cref{lem:split}, \[B_\Gamma\cong(B_2\times B_1\times B_1)\rtimes B_{1,2}\cong\Z\times B_{1,2}.\] So it suffices to show that any homomorphism $B_3\to B_{1,2}$ is cyclic. By \Cref{thm:bntob3}(iii), the image of any non-cyclic homomorphism $B_3\to B_3$ surjects onto $S_3$. Since $B_{1,2}<B_3$ does not surject onto $S_3$, it follows that $\varphi$ is cyclic.
\end{proof}

\subsubsection{The partition $3+1$}\ \medskip

\textbf{Claim:} If $\ocrs(\varphi)$ is $\Gamma\coloneqq\{\gamma_1,\gamma_2\}$ as depicted in \Cref{fig:4-fix-3-curve} up to the action of $B_4$, then $\varphi$ belongs to case (i) of \Cref{thm:B3-to-B4}.

\begin{proof}
  By \Cref{lem:split}, \[B_\Gamma\cong(B_3\times B_1)\rtimes B_{1,1}\cong B_3\times\Z.\] The $B_3$ factor can be identified with the standard copy of $B_3<B_4$. The $\Z$ factor produced by the section in \Cref{lem:split} is generated by $z_3z_4^{-1}$, where $z_3=(\sigma_1\sigma_2)^3$ generates the center of $B_3$ and $z_4=(\sigma_1\sigma_2\sigma_3)^4$ generates the center of $B_4$. Hence $B_\Gamma=B_3\times Z(B_4)$, and so $\varphi$ is a composition of an endomorphism of $B_3$, followed by the standard inclusion, up to transvections by a homomorphism $B_3\to Z(B_4)$ and automorphisms of $B_4$.
\end{proof}

In conclusion, if $\varphi\colon B_3\to B_4$ is reducible and non-cyclic, then it belongs to case (i) of \Cref{thm:B3-to-B4}, which completes the proof of \Cref{lem:b3-to-b4-reducible}.
\end{proof}

\subsection{The periodic case}\label{sec:other-cases}

\begin{lemma}\label{lem:b3-to-b4-other-cases}
  If $\varphi(z_3)$ is periodic, then $\varphi$ belongs to case (ii) of \Cref{thm:B3-to-B4}.
\end{lemma}

\begin{proof}
  We consider two cases.

\subsubsection{Periodic but not central}\ \medskip

\textbf{Claim:} If $\varphi(z_3)$ is periodic but not central, then $\varphi$ is cyclic.

\begin{proof}
  By \Cref{prop:periodic}, the centralizer of $\varphi(z_3)$ is isomorphic to $B_{d,1}$ where $d$ is a proper divisor of either 3 or 4, i.e., $d=1$ or $d=2$. If $d=1$, then $B_{d,1}\cong\Z$ and so $\varphi$ is cyclic. If $d=2$, then $B_{d,1}=B_{2,1}$. By \Cref{thm:bntob3}, the image of any non-cyclic endomorphism of $B_3$ surjects onto $S_3$. Thus any homomorphism $B_3\to B_{2,1}$ is cyclic. Therefore $\varphi$ is cyclic.
\end{proof}

\subsubsection{Central}\ \medskip

\textbf{Claim:} If $\varphi(z_3)$ is central, then $\varphi$ belongs to case (ii) of \Cref{thm:B3-to-B4}.

\begin{proof}
  Since $\varphi(Z(B_3))\leq Z(B_4)$, $\varphi$ descends to a homomorphism
  \[
    \psi\colon B_3/Z(B_3) \to B_4/Z(B_4).
  \]
  Let $a,b$ denote the images of $\alpha,\beta$ in $Z/Z(B_3)$. If $\psi(a)$ or $\psi(b)$ is trivial, then $\psi$ is cyclic, and hence $\varphi$ is cyclic, contradiction. Hence $\psi(a)$ has order 3 and $\psi(b)$ has order 2. By \Cref{prop:what-is-periodic}, possibly after replacing $\varphi$ with $\iota_4\circ\varphi$, we find that $\psi(a)$ and $\psi(b)$ are conjugate to the images of $\beta_4$ and $\alpha_4^2$ in $ B_4/Z(B_4)$, respectively. By the relation $\alpha^3=\beta^2$ it follows that, up to a transvection $B_3\to Z(B_4)$ and an automorphism of $B_4$, $\varphi$ belongs to case (ii) of \Cref{thm:B3-to-B4}.
\end{proof}

This completes the proof of \Cref{lem:b3-to-b4-other-cases}, and hence \Cref{thm:B3-to-B4}.
\end{proof}

\subsection{Proof of Theorem~\ref{thm:b3-to-b4-pseudoperiodic}}\label{sec:pseudo-periodic}

We now classify the irreducible homomorphisms $\varphi\colon B_3\to B_4$ such that $\varphi(\sigma_1)$ and $\varphi(\sigma_2)$ have some powers which are multitwists. This will be used in \Cref{sec:conf-spaces}. The following lemma imposes a strong constraint on the image of $\sigma_1$.

\begin{lemma}
  \label{lem:pseudo-periodic-constraint}
  Let $\varphi\colon B_3\to B_4$ be an irreducible, non-cyclic homomorphism such that $\varphi(\sigma_1)$ has a power that is a multitwist. Then there exist $\tau\in\Aut(B_4)$ and $t\colon B_3\to Z(B_4)$ such that $\tau\circ\varphi^t$ maps $\sigma_1\mapsto (\sigma_1\sigma_2)^{6k+1}$ for some $k\in\Z$.
\end{lemma}

\begin{proof}
  Since $\varphi$ irreducible and non-cyclic, without loss of generality by \Cref{thm:B3-to-B4} we may assume $\varphi$ is given by $\alpha\mapsto g\beta_4g^{-1}$, $\beta\mapsto\alpha_4^2$ for some $g\in B_4$. Thus $\varphi(\sigma_1)=g\beta_4^{-1}g^{-1}\alpha_4^2$. Suppose that $\varphi(\sigma_1)$ preserves no essential simple closed curve. Then $\varphi(\sigma_1)$ is periodic. However, the image of $\varphi(\sigma_1)$ has order 6 in the abelianization $\Z/12\Z$ of $B_4/Z(B_4)$, whereas periodic elements have orders dividing 3 or 4 in $B_4/Z(B_4)$. Therefore, $\varphi(\sigma_1)$ preserves some essential simple closed curve. 

Since $\varphi(\sigma_1)$ induces a permutation of the marked points that is a 3-cycle in $S_4$ then, up to conjugation in $B_4$, the canonical reduction system of $\varphi(\sigma_1)$ consists of the single curve $\gamma\coloneqq\gamma_1$ in \Cref{fig:4-fix-3-curve}. Using the isomorphism $B_{\{\gamma\}}\cong B_3\times B_{1,1}=B_3\times Z(B_4)$ given by \Cref{lem:split} we see that, up to transvections $B_3\to Z(B_4)$, some power of $\varphi(\sigma_1)$ is conjugate to a power of $\sigma_1\sigma_2$, as these are the only periodic elements of $B_3$ having order 3 in $B_3/Z(B_3)$. Since $\varphi(\sigma_1)$  has order 6 in the abelianization $\Z/12$ of $B_4/Z(B_4)$, it follows that $\varphi(\sigma_1)$ is conjugate to a power $(\sigma_1\sigma_2)^{6k+1}$ for some $k\in\Z$.
\end{proof}

It remains to determine the possible values of $\varphi(\sigma_2)\in B_4$ under the hypothesis that $\varphi(\sigma_1)=(\sigma_1\sigma_2)^{6k+1}$. To aid us in this we invoke the following result of Gassner \cite{Gas61}.

\begin{theorem}
  \label{thm:gassner}
  The kernel of $R_*\colon B_4\to B_3$ is freely generated by
  \[
    x\coloneqq\sigma_1^{-1}\sigma_3, \quad \text{ and } \quad y\coloneqq x^{-1}\sigma_2^{-1}x\sigma_2.
  \]
  The action of the subgroup $B_3<B_4$ on $\ker(R_*)$ by conjugation is given by
  \[
    \sigma_1^{-1}x\sigma_1 = x, \quad \sigma_2^{-1}x\sigma_2 = xy, \quad \sigma_1^{-1}y\sigma_1 = yx^{-1}, \quad \sigma_2^{-1}y\sigma_2 = y.
  \]
\end{theorem}

\begin{proof}
  Gassner shows that $\ker(R_*)$ is freely generated by $x^{-1}=\sigma_3^{-1}\sigma_1$ and $yx^{-1}=\sigma_2\sigma_1\sigma_3^{-1}\sigma_2^{-1}$ \cite[Theorem 7]{Gas61}, and gives an equivalent description of the action of $B_3$ on these generators.
\end{proof}

Since the standard inclusion $B_3\hookrightarrow B_4$ is a section of $R_*$ the short exact sequence
\[
  1 \to F_2 \to B_4 \xrightarrow{R_*} B_3 \to 1
\]
splits, where $F_2$ is the free group generated by $x$ and $y$. \Cref{thm:gassner} thus gives an explicit description of the semi-direct product decomposition of $B_4=F_2\rtimes B_3$.

Let $f\in\Aut(F_2)$ be defined by $w\mapsto(\sigma_1\sigma_2)^{-1}w(\sigma_1\sigma_2)$. By \Cref{thm:gassner} we have
\[
  f(x) = xy, \qquad f(y) = x^{-1}.
\]
In the proof of \Cref{thm:b3-to-b4-pseudoperiodic} we will show that classifying the irreducible, non-cyclic homomorphisms $\varphi\colon B_3\to B_4$ with $\varphi(\sigma_1)$ having some power that is a multitwist, is equivalent to solving the following delicate equation in $F_2$ involving the automorphism $f$
\[
  f^{6k+1}(w)f^{-6k-1}(w) = w, \quad w\in F_2.
\]
Solving this equation is the purpose of \Cref{sec:f2}, see \Cref{lem:f2-aut}. There we show that if there is a non-identity solution, then $k=0$ and $w=f^n(x)$ for some $n\in\Z$. 

\begin{proof}[Proof of \Cref{thm:b3-to-b4-pseudoperiodic}]
  Let $\varphi\colon B_3\to B_4$ be a non-cyclic, irreducible homomorphism such that some power of $\varphi(\sigma_1)$ is a multitwist. We seek $\tau\in\Aut(B_4)$ and $t\colon B_3\to Z(B_4)$ such that $\Psi_{3*}=\tau\circ\varphi^t$.

  By \Cref{lem:pseudo-periodic-constraint}, we can assume without loss of generality that $\varphi(\sigma_1)=(\sigma_1\sigma_2)^{6k+1}$ for some integer $k$. By \Cref{thm:bntob3}, $R_*\circ\varphi$ is cyclic, so there exists $w\in F_2$ such that $\varphi(\sigma_2)=w(\sigma_1\sigma_2)^{6k+1}$. The braid relation $\varphi(\sigma_1\sigma_2\sigma_1)=\varphi(\sigma_2\sigma_1\sigma_2)$ is equivalent to
  \[
    f^{6k+1}(w)f^{-6k-1}(w) = w.
  \]
  By \Cref{lem:f2-aut}, $k=0$ and $w=f^n(x)$ for some $n\in\Z$. Thus
  \[
    \varphi(\sigma_2) = f^n(x)(\sigma_1\sigma_2) = (\sigma_1\sigma_2)^{-n}(\sigma_1^{-1}\sigma_3)(\sigma_1\sigma_2)^{n+1} = (\sigma_1\sigma_2)^{-n}(\sigma_3\sigma_2)(\sigma_1\sigma_2)^n.
  \]
  Hence we may take $t$ to be trivial, and $\tau\colon g\mapsto(\sigma_1\sigma_2)^nx(\sigma_1\sigma_2)^{-n}$.
\end{proof}

\section{Solution to the equation in the free group of rank 2}
\label{sec:f2}

Let $x,y$ be a free basis of the free group $F_2$ of rank 2, and let $f\in\Aut(F_2)$ be the automorphism defined by
\[
  f(x) \coloneqq xy, \qquad f(y) \coloneqq x^{-1}.
\]

The goal of this section is to prove the following theorem.

\begin{theorem}
  \label{lem:f2-aut}
  If $w\in F_2$ is a non-identity word satisfying
  \[
    f^{6k+1}(w)f^{-6k-1}(w) = w,
  \]
  then $k=0$ and $w=f^n(x)$ for some $n\in\Z$.
\end{theorem}

\subsection{Basic properties of $f$}

Note that $f$ has order 6 in $\GL_2\Z$. In fact:
\[
  f^6(w) = cwc^{-1}, \quad \text{where } c\coloneqq[x,y]=xyx^{-1}y^{-1}.
\]

\begin{lemma}
  \label{lem:f-fixed}
  The set of words fixed by $f$ is the subgroup $\langle c\rangle$.
\end{lemma}

\begin{proof}
  One can easily check that $f(c)=c$, hence $f$ fixes every element of $\langle c\rangle$. Conversely, if $w$ is fixed by $f$, then it is also fixed by $f^6$. Hence $w$ must commute with $c$. Since $c$ has cyclic centralizer and is not an $n$th power in $F_2$ for any $n\geq2$, it follows that $w\in\langle c\rangle$.
 \end{proof}

If $w\in F_2$, then let $\ell(w)$ be the length of $w$ as a reduced word in the alphabet $\{x,x^{-1},y,y^{-1}\}$.

\begin{definition}[\textbf{Start and end}]
  Suppose that $w=uv$ where $u,v\in F_2$ and there is no cancellation in the product $uv$, i.e., $\ell(w)=\ell(u)+\ell(v)$. Then we say that $w$ \emph{starts with} $u$ and \emph{ends with} $v$.

  We write $w=u\gap$ to denote that $w$ starts with $u$, or $w=\gap v$ to denote that $w$ ends with $v$, or $w=u\gap v$ to denote that $w$ starts with $u$ and ends with $v$.
\end{definition}

For example, we may write $c=xy\gap$, or $c=\gap y^{-1}$. It is also true that $c=xyx^{-1}\gap x^{-1}y^{-1}$, since in the expression $w=u\gap v$, the subwords $u$ and $v$ can overlap in the expression of $w$ as a reduced word.

The following lemma gives a precise relationship between the start and end of $w$ and $f(w)$.

\begin{lemma}
  \label{lem:start-end}
  Let $n\in\Z$. Then
   \begin{alignat*}{2}
    w &= y^nx\gap &\iff f(w) &= x^{-n+1}y\gap \\
    w &= y^nx^{-1}\gap & \iff f(w) &= x^{-n}y^{-1}\gap \\
    w &= \gap xy^n & \iff f(w) &= \gap yx^{-n} \\
    w &= \gap x^{-1}y^n & \iff f(w) &= \gap y^{-1}x^{-n-1}.
  \end{alignat*}
\end{lemma}

\begin{proof}
  Write
  \[
    w=y^{n_0}x^{m_1}y^{n_1}\cdots x^{m_r}y^{n_r},
  \]
  where $r\geq1$ and the exponents all nonzero except possibly $n_0$ and $n_r$. We will prove all four equivalences simultaneously by induction on $r$. We have
  \[
    f(y^{n_0}x^{m_1}y^{n_1}) = x^{-n_0}(xy)^{m_1}x^{-n_1} = \begin{cases} x^{-n_0+1}y\gap yx^{-n_1} & m_1>0 \\ x^{-n_0}y^{-1}\gap y^{-1}x^{-n_1-1} & m_1<0 \end{cases}
  \]
  which proves the lemma for $r=1$. Suppose now that $r>1$, and let
  \[
    w_1 \coloneqq y^{n_0}x^{m_1}y^{n_1}\cdots x^{m_{r-1}}y^{n_{r-1}}, \qquad w_2 \coloneqq x^{m_r}y^{n_r},
  \]
  so that $w=w_1w_2$. By the inductive hypothesis
  \[
    f(w_1) = \begin{cases} \gap yx^{-n_{r-1}} & m_{r-1} > 0 \\ \gap y^{-1}x^{-n_{r-1}-1} & m_{r-1} < 0, \end{cases}
  \]
  and
  \[
    f(w_2) = \begin{cases} xy\gap & m_r > 0 \\ y^{-1}\gap & m_r<0 \end{cases}
  \]
  We claim that if there is cancellation in the product $f(w_1)f(w_2)$, then only one letter from each factor cancels: an $x^{-1}$ at the end of $f(w_1)$, and an $x$ at the start of $f(w_2)$.

Note that $n_{r-1}\neq0$ since $r>1$. If $f(w_1)$ ends with a non-trivial power of $y$, then it must end with $y^{-1}$. However $f(w_2)$ does not start with $y$. Hence, if there is cancellation, then $f(w_1)$ must end with a non-trivial power of $x$. If $f(w_2)$ starts with a non-trivial power of $x$, then $m_r>0$, in which case $f(w_2)=xy\gap$. If there is cancellation, then $f(w_1)$ must end with $x^{-1}$. To cancel two letters from each factor would require $f(w_1)$ to end with $y^{-1}x^{-1}$, which is not possible since $n_{r-1}\neq0$.
  
In particular, no powers of $y$ cancel in the product $f(w_1)f(w_2)$. Hence, if we apply the inductive hypothesis to the start of $w_1$ and the end of $w_2$, then we find that all four equivalences also hold for $w$.
\end{proof}

\subsection{Proof that $k=0$}

Suppose that $w\notin\langle c\rangle$. Let $L(w),R(w)$ be maximal in absolute value such that $w=c^{L(w)}\gap c^{-R(w)}$, and let $m(w)\in F_2$ be the unique word satisfying
\[
  w=c^{L(w)}m(w)c^{-R(w)}.
\]
\begin{lemma}
Suppose that $w\notin\langle c\rangle$. Then there is no cancellation in the product $c^{L(w)}m(w)c^{-R(w)}$, i.e., $\ell(w)=\ell(c^{L(w)})+\ell(m(w))+\ell(c^{-R(w)})$.
\end{lemma}
\begin{proof}
This is equivalent to showing that the $c^{L(w)}$ at the start of $w$ does not overlap with the $c^{-R(w)}$ at the end of $w$.

Suppose for a contradiction that they do overlap. Then there exists $u\in F_2$ with $0<\ell(u)<\ell(c)$ such that $c^\epsilon=\gap u$ and $c^{\epsilon'}=u\gap$, where $\epsilon,\epsilon'\in\{\pm1\}$ depend on the signs of $L(w)$ and $R(w)$. If $\epsilon=-\epsilon'$, then $u=u^{-1}$ which is impossible. If $\epsilon=\epsilon'$, then $c^\epsilon=u\gap u$, however the only non-trivial word $u$ satisfying this is $u=c^\epsilon$. Hence there is no overlap.
\end{proof}

\begin{lemma}
  \label{lem:free-group-product}
  If $w_1,w_2,w_1w_2\notin\langle c\rangle$, then either $L(w_1w_2)=L(w_1)$, or $R(w_1w_2)=R(w_2)$.
\end{lemma}

\begin{proof}
  Write $L_i=L(w_i)$, $R_i=R(w_i)$, $m_i=m(w_i)$, $L_{12}=L(w_1w_2)$, and $R_{12}=R(w_1w_2)$. Then
  \[
    w_1w_2 = c^{L_1}m_1c^{L_2-R_1}m_2c^{-R_2}.
  \]
The only way for both $L_{12}\neq L_1$ and $R_{12}\neq R_2$ to hold is for the middle factor $c^{L_2-R_1}$ to be trivial, and for only three letters of $m_1$ and three letters of $m_2$ to remain in $m_1m_2$ after cancellation. But since $\ell(c)=4$, this leaves only enough letters to form at most $c^{\pm1}$, so at most one power of $c$ from one end can be altered. Hence, either $L_{12}=L_1$ or $R_{12}=R_2$.
\end{proof}

\begin{lemma}
  \label{lem:non-decreasing}
  If $w\notin\langle c\rangle$ then $L(w)\leq L(f(w))$ and $R(w)\leq R(f(w))$.
\end{lemma}

\begin{proof}
  It suffices to prove the first inequality for all $w\notin\langle c\rangle$, since $R(w)=L(w^{-1})$.

  If $L(w)>0$, then
  \[
    w = c^{L(w)-1}xyx^{-1}\gap,
  \]
  which by \Cref{lem:f-fixed,lem:start-end} implies that $f(w)=c^{L(w)}\gap$, so $L(w)\leq L(f(w))$.

  Similarly, if $L(f(w))<0$, then
  \[
    f(w) = c^{L(f(w))+1}yxy^{-1}\gap,
  \]
  which by \Cref{lem:f-fixed,lem:start-end} implies that $w=c^{L(f(w))}\gap$, so $L(w)\leq L(f(w))$.

  The only other possibility is $L(w)\leq0\leq L(f(w))$, and the inequality clearly holds here.
\end{proof}

\begin{lemma}
  \label{lem:k-is-0}
  If there is a non-identity word $w$ satisfying $w = f^{6k+1}(w)f^{-6k-1}(w)$, then $k=0$.
\end{lemma}

\begin{proof}
  Note that the set of solutions to this equation is closed under the action of $\langle f^{6k+1}\rangle$. Let $w$ be a non-identity solution to $w=f^{6k+1}(w)f^{-6k-1}(w)$. The identity is the only solution that is fixed by $f$, so $w\notin\langle c\rangle$ by \Cref{lem:f-fixed}.

  Write $w_n=f^n(w)$, $L_n=L(w_n)$ and $R_n=R(w_n)$. Since $w_{n+6}=cw_nc^{-1}$, then for all $n$ sufficiently large $L_{n+6}=L_n+1$ and $R_{n+6}=R_n+1$. Furthermore, the sequences $L_n$ and $R_n$ are non-decreasing by \Cref{lem:non-decreasing}.
Since $6k+1$ is coprime to 6, there exists $n$ such that the following hold: 
  \begin{itemize}
  \item $L_{n+1}=L_n$
  \item $L_{n+6k+1}=L_{n+1}+k$
  \item $R_{n-1}=R_n$
  \item $R_{n-6k-1}=R_{n-1}-k$.
  \end{itemize}
  In particular $L_{n+6k+1}=L_n+k$ and $R_{n-6k-1}=R_n-k$. By \Cref{lem:free-group-product} and the equation $w_n=w_{n+6k+1}w_{n-6k-1}$ we can conclude that either $L_n=L_n+k$ or $R_n=R_n-k$. In either case we must have $k=0$.
\end{proof}

\subsection{The solutions when $k=0$}

\begin{lemma}
  \label{lem:same-start-end}
  Suppose that both $w$ and $f(w)$ start with $u$, where $\ell(u)=4$. Then $u=c^{\pm1}$. Similarly if both $w$ and $f(w)$ end with $v$, where $\ell(v)=4$, then $v=c^{\pm1}$.
\end{lemma}

\begin{proof}
  Note that the second statement follows from applying the first to $w^{-1}$. Write
  \[
    w=y^{n_0}x^{m_1}y^{n_1}\cdots x^{m_r}y^{n_r}
  \]
  as before. Suppose that $w$ and $f(w)$ both start with the same length four word $u$. By \Cref{lem:start-end}, $w$ starts with either $yx$ or $x$. Hence $r>0$, $m_1>0$ and $n_0\in\{0,1\}$. In the arguments that follow, we will freely apply \Cref{lem:start-end} in similar ways to the previous application. For instance,
  \[
    f(y^{n_1}\cdots x^{m_r}y^{n_r}) = \begin{cases} x^{-n_1} & k=1 \\ x^{-n_1+1}y\gap & m_2 > 0 \\ x^{-n_1}y^{-1}\gap & m_2 < 0. \end{cases}
  \]
  Hence,
  \[
    f(w) = \begin{cases} x^{-n_0+1}y(xy)^{m_1-1}x^{-n_1} & r = 1 \\ x^{-n_0+1}y(xy)^{m_1-1}x^{-n_1+1}y \gap & m_2 > 0 \\ x^{-n_0+1}y(xy)^{m_1-1}x^{-n_1}y^{-1}\gap & m_2<0. \end{cases}
  \]
  If $n_0=0$, then $f(w)=xy\gap$, hence $m_1=1$.
  If $n_0=1$, then $f(w)=y(xy)^{m_1-1}\gap$, hence $m_1$ cannot be larger than 1. Thus, in all cases $m_1=1$. Hence,
  \[
    f(w) = \begin{cases} x^{-n_0+1}yx^{-n_1} & r = 1 \\ x^{-n_0+1}yx^{-n_1+1}y \gap & m_2 > 0 \\ x^{-n_0+1}yx^{-n_1}y^{-1}\gap & m_2<0. \end{cases}
  \]
  Clearly $r=1$ leads to a contradiction. If $n_0=1$, then
  \[
    f(w) = \begin{cases} yx^{-n_1+1}y\gap & m_2 > 0 \\ yx^{-n_1}y^{-1}\gap & m_2 < 0. \end{cases}
  \]
  Since $m_1=1$ (and $n_0=1$) we must have $n_1=-1$ and $m_2<0$. Thus, if $n_0=1$, then $w$ starts with $c^{-1}=yxy^{-1}x^{-1}$.

  Now suppose $n_0=0$, and assume for a contradiction that $n_1\neq1$. Then $f(w)=xyx^{e}\gap$, where $e\neq0$. Hence, by comparing with $w$ we must have $n_1=1$.

  Thus, if $n_0=0$, then $m_1=n_1=1$, and hence $m_2<0$. Note that
  \[
    f(y^{n_2}\cdots x^{n_k}y^{n_k}) = \begin{cases} x^{-n_2} & r = 2 \\ x^{-n_2+1}y\gap & m_3 > 0 \\ x^{-n_2}y^{-1}\gap & m_3 < 0. \end{cases}
  \]
  Hence, if $n_0=0$, then
  \[
    f(w) = \begin{cases} xyx^{-1}(y^{-1}x^{-1})^{-m_2-1}y^{-1}x^{-n_2-1} & r = 2 \\ xyx^{-1}(y^{-1}x^{-1})^{-m_2-1}y^{-1}x^{-n_2}y\gap & m_3 > 0 \\ xyx^{-1}(y^{-1}x^{-1})^{-m_2-1}y^{-1}x^{-n_2-1}y^{-1}\gap & m_3<0, \end{cases}
  \]
  and so $w$ starts with $c=xyx^{-1}y^{-1}$.
\end{proof}

In what follows, the \emph{$f$-orbit of $w$} is the set $\{f^n(w) : n\in\Z\}$, i.e., the union of the forwards and backwards $f$-orbits of $w$.

\begin{lemma}
  \label{lem:consecutive}
  Every word in $F_2$ not in $\langle c\rangle$ contains a word $w$ in its $f$-orbit such that $f^{-1}(w)$, $w$, and $f(w)$ all neither end with $c$ nor $c^{-1}$.
\end{lemma}

\begin{proof}
  Suppose $w\notin\langle c\rangle$. Since $f^6(w)=cwc^{-1}$, the word $f^n(w)$ ends with $c^{-1}$ for all sufficiently large $n$, and ends with $c$ for all sufficiently negative $n$.

  For a given $f$-orbit not in $\langle c\rangle$, choose $v$ in it such that $v$ ends in $c$ but $f(v)$ does not, and let $w\coloneqq f^2(v)$, so $f^{-1}(w)$ does not end in $c$. By \Cref{lem:start-end}, neither $w$ nor $f(w)$ end with $c$. Also, by \Cref{lem:start-end},
  \[
    v = \gap c \implies f(v) = f^{-1}(w) = \gap yx^{-1}y^{-1} \implies w = \gap y^{-1}.
  \]
  Suppose for a contradiction that $f(w)=\gap c^{-1}$. Then in particular $f(w)=\gap y^{-1}x^{-1}$, so by \Cref{lem:start-end}
  \[
    w = \gap x^{-1},
  \]
  which is a contradiction. Hence, none of $f^{-1}(w)$, $w$ or $f(w)$ end in $c^{-1}$.
\end{proof}

\begin{lemma}
  \label{lem:f-orbit-x}
  If $w$ is a non-identity solution of $w=f(w)f^{-1}(w)$, then $w=f^n(x)$ for some $n\in\Z$.
\end{lemma}

\begin{proof}
  First note the claim holds under the additional hypothesis that $\ell(w)\leq6$, which is easily verified by a computer. We can also ignore elements of $\langle c\rangle$, since the only solution contained in there is the identity.
  
  Suppose $w$ neither starts nor ends with either of $c$ or $c^{-1}$. Then by \Cref{lem:same-start-end}, in the product $w=f(w)f^{-1}(w)$ all but at most the first three letters of $f(w)$ and at most the last three letters of $f^{-1}(w)$ must be cancelled, hence $\ell(w)\leq6$, so $w$ is in the $f$-orbit of $x$.

  Assume then that all words in the $f$-orbit of $w\notin\langle c\rangle$ either start or end with one of $c$ or $c^{-1}$. By \Cref{lem:consecutive}, we may assume without loss of generality that $f^{-1}(w)$, $w$ and $f(w)$ do not end with either of $c$ or $c^{-1}$, and hence all start with $c$ or $c^{-1}$.

  In particular, by \Cref{lem:same-start-end} all but at most three letters of $f^{-1}(w)$ are cancelled in the product $f(w)f^{-1}(w)$. Hence $\ell(f^{-1}(w))\leq6$, so $w$ is in the $f$-orbit of $x$.
\end{proof}

\begin{proof}[Proof of \Cref{lem:f2-aut}]
  If $w$ is a non-identity solution of $f^{6k+1}(w)f^{-6k-1}(w)$, then $k=0$ by \Cref{lem:k-is-0}, and $w=f^n(x)$ for some $n\in\Z$ by \Cref{lem:f-orbit-x}.
\end{proof}

\section{Classification of holomorphic maps}\label{sec:conf-spaces}

The goal of this section is to classify the holomorphic maps $\Conf_n\C\to\Conf_m\C$ for $m\in\{3,4\}$, and to classify the families of elliptic curves over $\Conf_n\C$.

We denote by $\M_{g,n}$ the moduli space of Riemann surfaces of genus $g$ with $n$ unordered marked points. We only consider this moduli space when $2g-2+n>0$, where it can be identified with the orbifold quotient $\Mod_{g,n}\backslash{\mathcal T_{g,n}}$, where $\Mod_{g,n}$ and ${\mathcal T_{g,n}}$ are the associated mapping class group and Teichm{\"u}ller spaces. Note that $\Mod_{g,n}$ need not act faithfully on ${\mathcal T_{g,n}}$, e.g., $\M_{1,1}=\SL_2\Z\backslash\HH$. Adopting this convention means that holomorphic maps $B\to\M_{g,n}$, in the orbifold sense, where $B$ is a complex manifold, correspond to families of genus $g$ curves having $n$ marked points.

Following the techniques of Chen and Salter \cite{CS23}, we first constrain the possible homomorphisms that can come from holomorphic maps, and then apply a theorem of Imayoshi and Shiga \cite{IS88} to see that each possible equivalence class is realized by a unique holomorphic map, up to affine twists.

\subsection{Constraining the homomorphisms}

The following is a well-known consequence of work of Bers \cite{Ber78}, see, e.g., \cite[\S3]{McM00}.

\begin{theorem}
  \label{thm:punctured-disk-family}
  Let $f\colon\D^*\to\M_{g,n}$ be holomorphic, where $\D^*=\{z\in\C : 0 < |z| < 1\}$ is a punctured disk and $2g-2+n>0$. Then $f_*(k)$ is a multitwist for some $k\geq1$, where $f_*\colon\Z\to\Mod_{g,n}$ is the induced map on fundamental groups.
\end{theorem}

\begin{proof}
  It is a theorem of Royden \cite{Roy71} (see \cite[\S2.3]{EK74}) that $\M_{g,n}$ is Kobayashi hyperbolic, with Kobayashi metric equal to the Teichm\"{u}ller metric. Hence, the holomorphic map $f$ is distance non-increasing with respect to the hyperbolic metric on $\D^*$ and the Teichm\"{u}ller metric on $\M_{g,n}$. Since a loop around the puncture in $\D^*$ can be made arbitrarily short in the hyperbolic metric, $f_*(k_0)$ has translation length zero for all $k_0\in\Z$. Let $k_0\geq1$ be such that $f_*(k_0)$ preserves the connected components of the complement of its canonical reduction system. By assertion (A) in the proof of Bers' Theorem 7 \cite{Ber78}, we can conclude that $f_*(k_0)$ is not pseudo-Anosov when restricted to any of these connected components. Thus, $f_*(k_0)$ is periodic on each component, hence $f_*(k)$ is a multitwist for some positive multiple $k$ of $k_0$.
\end{proof}

As a useful consequence we have the following.

\begin{proposition}
  \label{prop:multitwist}
  Let $f\colon\Conf_m\C\to\M_{g,n}$ be a holomorphic map, where $2g-2+n>0$. Then there exists some $k\neq0$ such that $f_*(\sigma_i^k)$ is a multitwist for $i=1,\ldots,n-1$.
\end{proposition}

\begin{proof}
  The configuration space $\Conf_m\C$ is isomorphic to the space of monic degree $m$ square-free polynomials over $\C$. This identifies $\Conf_m\C$ with $\C^m-{\mathcal D}_m$, where ${\mathcal D}_m$ is the discriminant locus. The smooth points of ${\mathcal D}_m$ correspond to configurations with exactly one repeated point. Thus, a loop in the discriminant complement representing $\sigma_i$ is freely homotopic to a small loop $\gamma$ in the normal space of ${\mathcal D}_m$ at some smooth point. In fact, we can choose the same $\gamma$ for each $\sigma_i$, since the set of smooth points of ${\mathcal D}_m$ is connected. Applying \Cref{thm:punctured-disk-family} to the punctured disk bounded by $\gamma$ gives the result.
\end{proof}

We will also need the following theorem, which is a special case of a theorem of Daskalopoulos and Wentworth.

\begin{theorem}[{\cite[Theorem 5.7]{DW07}}]
  \label{thm:DW07}
  Let $f\colon B\to\M_{g,n}$ be a holomorphic map, where $B$ is a Riemann surface of finite type. If $f_*\colon\pi_1(B)\to\Mod_{g,n}$ is cyclic or reducible, then $f$ is constant.
\end{theorem}

Chen and Salter prove the following in \cite[\S3.3]{CS23}.

\begin{lemma}
Let $f:\Conf_n\C\to \Conf_m\C$ be a holomorphic map such that the induced map $\bar{f}:\Conf_n\C \to \Conf_m\C/\mathrm{Aff}$ is constant. Then $f$ is an affine twist of a root map.
\end{lemma}

\subsection{Uniqueness}

The following result is a generalization of a theorem of Imayoshi--Shiga \cite{IS88}. We mimic the proof of \cite[Proposition 3.2]{AAS18} and invoke \cite[Theorem 2.5]{CS23}.

\begin{theorem}
  \label{thm:AAS18}
  Let $M$ be a smooth, quasi-projective complex variety, let $\M_{g,n}'$ be a finite orbifold cover of $\M_{g,n}$, and let $f,g\colon M\to\M_{g,n}'$ be non-constant holomorphic or antiholomorphic maps. If $f$ and $g$ are homotopic as orbifold maps, then $f=g$.
\end{theorem}

\begin{proof}
  Chen and Salter's result \cite[Theorem 2.5]{CS23} handles the case where $M$ is a Riemann surface of finite type. For the general case, note that a generic curve in $M$ is a Riemann surface of finite type. Thus, $f$ and $g$ agree on a generic curve in $M$, and hence $f=g$.
\end{proof}

Note that $\Conf_m\C$ is a smooth quasi-projective variety, so \Cref{thm:AAS18} applies to it. Furthermore, since $\Conf_m\C$ is a $K(\pi,1)$ space, $f,g\colon Y\to\M_{g,n}'$ are homotopic as orbifold maps if and only if $f_*,g_*\colon\pi_1(Y)\to\pi_1(\M_{g,n}')$ are conjugate, for any covering space $Y$ of $\Conf_m\C$.

The following consequence is proved by Chen and Salter in the special case where $m=n$ \cite[\S3.1]{CS23}. Note that there is a natural holomorphic map $\Conf_n\C\to\M_{0,n+1}$, given by associating to $\{x_1,\ldots,x_n\}\in\Conf_n\C$ the Riemann surface $\CP^1=\C\cup\{\infty\}$, with marked points $\{x_1,\ldots,x_n,\infty\}$.

\begin{proposition}
  \label{prop:grothendieck}
  Let $f,g\colon\Conf_m\C\to\Conf_n\C$ be holomorphic maps such that the induced maps $\bar{f},\bar{g}\colon\Conf_m\C\to\M_{0,n+1}$ are not constant. Suppose that $f_*,g_*\colon B_m\to B_n$ differ by at most an automorphism of $B_n$ and a transvection $B_m\to Z(B_n)$. Then $f$ is an affine twist of $g$.
\end{proposition}

\begin{proof}
  Without loss of generality, we can apply an affine twist by a power of the discriminant $\Delta\colon\Conf_m\C\to\C^*$ so that $f_*,g_*$ differ only by an automorphism of $B_n$.

  Dyer and Grossman showed that $\Out(B_n)\cong\langle\iota\rangle$, where $\iota$ is given by $\sigma_i\mapsto\sigma_i^{-1}$ for $i=1,\ldots,n-1$ \cite{DG81}. The involution $\iota$ is induced by the coordinate-wise complex conjugation on $\Conf_n\C$. Thus, if $f_*=\tau\circ g_*$ for an outer automorphism $\tau$ of $B_n$, then by \Cref{thm:AAS18} the induced maps $\Conf_m\C\to\M_{0,n+1}$ only differ by complex conjugation, which contradicts both $f,g$ being holomorphic. Hence $f_*,g_*$ are conjugate, and thus by \Cref{thm:AAS18} they induce the same maps $\bar{f}=\bar{g}\colon\Conf_m\C\to\M_{0,n+1}$.

  Let $Y$ be the finite cover of $\Conf_m\C$ such that $f,g$ lift to maps $\tilde{f},\tilde{g}\colon Y\to\PConf_n\C$, and write $\tilde{f}=(f_1,\ldots,f_n)$ and $\tilde{g}=(g_1,\ldots,g_n)$, where $f_i,g_i\colon Y\to\C$ for $i=1,\ldots,n$. Note that $Y$ is also a quasi-projective variety by a result of Grauert and Remmert \cite{GR58}. Since $\bar{f}=\bar{g}$ and $f_*,g_*$ are conjugate, $\tilde{f}_*,\tilde{g}_*$ are also conjugate. \Cref{thm:AAS18} implies that $\tilde{f},\tilde{g}$ induce identical maps $Y\to\PConf_n\C/\Aff$. Hence, there exists a function $\tilde{A}\colon Y\to\Aff$ such $\tilde{f}=\tilde{g}^{\tilde{A}}$.

  If $\tilde{A}\cdot z=az+b$ for all $z\in\C$, where $a,b\colon Y\to\C$, then for all distinct $i,j\in\{1,\ldots,n\}$ we have
  \[
    a = \frac{g_i-g_j}{f_i-f_j}, \qquad b = \frac{f_ig_j-f_jg_i}{f_i-f_j}.
  \]
  Hence $\tilde{A}$ is holomorphic. Furthermore, the above holds for all distinct indices $i,j$, and $\tilde{f},\tilde{g}$ are equivariant with respect to the group of deck transformations of the covering map $Y\to\Conf_m\C$, which is a subgroup of $S_n$. Hence $\tilde{A}$ descends to a holomorphic map $A\colon\Conf_m\C\to\Aff$ such that $f=g^A$.
\end{proof}

\begin{remark}
  Grauert and Remmert's result is not needed by Chen and Salter, since in the case $m=n$ all such holomorphic maps lift to self maps of $\PConf_n\C$, which is clearly quasi-projective. However, we will apply \Cref{prop:grothendieck} with the map $\Psi_3$, which does not lift to a holomorphic map $\PConf_3\C\to\PConf_4\C$. Neither the argument of Grauert and Remmert, nor Grothendieck's vast generalization \cite[XII, Th{\'e}or{\`e}me 5.1]{Gro71}, give explicit equations for an arbitrary finite cover of $\Conf_n\C$. However, since $\Psi_3$ is also algebraic, this enables us to write down explicit equations for the pullback $Y\to\Conf_3\C$ of the cover $\PConf_4\C\to\Conf_4\C$ along $\Psi_3$ that express $Y$ as a quasi-projective variety.
\end{remark}

\subsection{Classifying the holomorphic maps}

We are finally ready to prove \Cref{thm:main-theorem}.

\begin{proof}[Proof of \Cref{thm:main-theorem}]
  Let $f\colon\Conf_n\C\to\Conf_m\C$ be a holomorphic map, where $n\geq3$ and $m\in\{3,4\}$. If $f$ is not a root map or the identity, then by \Cref{prop:multitwist} and \Cref{thm:DW07} the homomorphism $f_*\colon B_n\to B_m$ must be irreducible, non-cyclic, and $f_*(\sigma_i)$ must have a power that is a multitwist for all $i=1,\ldots,n-1$.

  If $m=3$, then for $f_*$ to be non-cyclic we must have $n\in\{3,4\}$ by \Cref{thm:bntob3}. Since $R_*$ is surjective, $n\in\{3,4\}$ is also necessary if $m=4$.

  Suppose $n=3$. By \Cref{thm:b3-to-parabolic} and \Cref{thm:b3-to-b4-pseudoperiodic}, $f_*$ is either the identity on $\Conf_3\C$ or $\Psi_{3*}$, up to a transvection $B_3\to Z(B_m)$ and an automorphism of $B_m$.

  Suppose $n=4$. By \Cref{thm:bntob3,thm:B4-to-B4} and the above, $f_*$ is either $R_*$, the identity on $\Conf_4\C$, or $\Psi_{3*}\circ R_*$, up to a transvection $B_4\to Z(B_m)$ and an automorphism of $B_m$.

  Since $f$ is not the identity, then $f$ is an affine twist of $R$, $\Psi_3$, or $\Psi_3\circ R$ by \Cref{prop:grothendieck}.
\end{proof}

Next we classify the holomorphic maps $\Conf_n\C\to\PSL_2\Z\backslash\HH$. There is a natural map $\SL_2\Z\backslash\HH\to\PSL_2\Z\backslash\HH$ corresponding to the quotient $\SL_2\Z\twoheadrightarrow\PSL_2\Z$. Moreover, $\PSL_2\Z\backslash\HH$ is the moduli space of quotients of elliptic curves by their hyperelliptic involution. This is the same data as a four times marked sphere, where one of the four points is distinguished. Hence $\PSL_2\Z\backslash\HH$ is a finite orbifold cover of $\M_{0,4}$.

\begin{theorem}
  \label{thm:hol-to-H/PSL2Z}
  Let $f\colon\Conf_n\C\to\PSL_2\Z\backslash\HH$ be a holomorphic map.
  \begin{itemize}
  \item If $n\geq5$, then $f$ is constant.
  \item If $n=4$, then $f=g\circ R$, where $g\colon\Conf_3\C\to\PSL_2\Z\backslash\HH$ is holomorphic.
  \item If $n=3$, then either $f$ is constant, or it is the post-composition $\bar{H}$ of the hyperelliptic embedding $H$ with the natural map $\M_{1,1}\cong\SL_2\Z\backslash\HH\to\PSL_2\Z\backslash\HH$.
  \end{itemize}
\end{theorem}

\begin{proof}
  Suppose that $f$ is non-constant. By \Cref{thm:DW07}, $f_*\colon B_n\to\PSL_2\Z$ is non-cyclic, and hence by \Cref{thm:BntoG}, $n<5$, and if $n=4$, then $f_*$ factors through $R_*$. By \Cref{prop:multitwist}, the $f_*(\sigma_i)$ are not hyperbolic. Let $\iota\in\Aut(B_n)$ be the inversion automorphism which is defined by $\sigma_i\mapsto\sigma_i^{-1}, i=1,\cdots,n-1$. By \Cref{thm:b3-to-parabolic} $n=3$ or $n=4$, and $f_*$ is $H_*$ up to precomposition by one or both of $R_*$ and $\iota$.

  Note that coordinate-wise complex conjugation is an antiholomorphic self map of $\Conf_n\C$ inducing $\iota$. Hence, by \Cref{thm:AAS18} either $f=\bar{H}$, $f=R\circ\bar{H}$, or $f$ is one of these maps precomposed with coordinate-wise complex conjugation.

  Since complex conjugation is antiholomorphic, the only holomorphic possibilities for $f$ are $f=H$ when $n=3$, and $f=H\circ R$ when $n=4$, as desired.
\end{proof}

\begin{proof}[Proof of \Cref{thm:family-elliptic-curves}]
  Let $f\colon\Conf_n\C\to\M_{1,1}$ be a holomorphic family such that the induced map $\bar{f}\colon\Conf_n\C\to\PSL_2\Z\backslash\HH$ non-constant. Since the kernel of $\SL_2\Z\twoheadrightarrow\PSL_2\Z$ is $Z(\SL_2\Z)\cong\Z/2$, and $H^1(B_n;\Z/2)\cong\Z/2$, each holomorphic map $\Conf_n\C\to\PSL_2\Z\backslash\HH$ lifts to at most two holomorphic families $\Conf_n\C\to\M_{1,1}$, whose induced homomorphisms $B_n\to\SL_2\Z$ differ by a transvection $B_n\to Z(\SL_2\Z)$. For any holomorphic family, we can induce such a transvection by pre-composing with $\id^{\Delta}$, since $\Delta\colon\Conf_n\C\to\C^*$ induces the abelianization map $\Delta_*\colon B_n\to\Z$.

  Hence, if $\bar{f}$ is non-constant, then by \Cref{thm:hol-to-H/PSL2Z} we have $n\in\{3,4\}$ and $f=H$ up to pre-composition by one or both of $R$ and $\id^{\Delta}$.
\end{proof}

\printbibliography

\end{document}